\documentclass{article}

\author{author}
\date{\today}
\title{title}

\usepackage[utf8]{inputenc}
\usepackage[T1]{fontenc}
\usepackage[english]{babel}
\usepackage{amsmath,amssymb}
\usepackage{bbm}
\usepackage{amsthm}
\usepackage{amssymb} 
\usepackage[mathscr]{eucal} 
\usepackage[matrix,arrow,tips,curve,ps]{xy}
\usepackage{amscd}
\usepackage{xcolor}
\usepackage{listings}
\usepackage{tikz}
\usepackage{tikz-cd}
\usepackage{mathrsfs}
\usepackage{floatflt,epsfig}
\usepackage{comment}
\usepackage{hhline}
\usetikzlibrary{backgrounds}
\usepackage{enumitem}
\usepackage[a4paper,top=3cm,bottom=3cm,left=3.5cm,right=3.5cm,%
bindingoffset=5mm]{geometry}
\usepackage{lipsum}       
\usepackage{changepage}   

\usepackage{url}
\usepackage{hyperref}

\usepackage{multirow}
\usepackage{arydshln}

\theoremstyle{plain} 
\newtheorem{thm}{Theorem}[section] 
\newtheorem*{thm*}{Theorem}
\newtheorem{cor}[thm]{Corollary} 
 
\newtheorem{prop}[thm]{Proposition} 
\theoremstyle{definition} 
\newtheorem{defn}{Definition}
\newtheorem*{qst*}{Question}
\newtheorem*{prop*}{Proposition}
\newtheorem*{cor*}{Corollary}
\theoremstyle{remark} 
\newtheorem{rmk}{Remark} 
\theoremstyle{remark} 
\newtheorem{claim}{Claim}

\newcommand{\D}[1]{{\mathbb#1}}

\newcommand{\C}{{\mathcal{C}}}

\newcommand{\E}{{\mathcal{E}}}
\newcommand{\R}{{\mathcal{R}}}
\newcommand{\N}{{\mathcal{N}}}

\newcommand{\pres}{\phi_{-K_{\tilde{S}}}}
\newcommand*{\sheafhom}{\mathscr{H}\kern -.5pt om}
\newcommand*{\sheafext}{\mathscr{E}\kern -.5pt xt}
\newcommand{\F}{{\mathcal{F}}}

\newcommand{\mapstoo}[1]{\:
	\xymatrix@1{\ar@{|->} [r] ^-{#1}&}\:}
\newlist{myenumerate}{enumerate}{1}
\setlist[myenumerate,1]{label=(\roman*)}

\newcommand{\DMN}[1]{\textcolor{black}{#1}}

\newcommand{\DM}[1]{\textcolor{black}{#1}}

\renewcommand{\restriction}{|}

\renewcommand\emptyset{\varnothing}

\newcommand{\Bl}{\textup{Bl}}
\newcommand{\codim}{\textup{codim\,}}
\newcommand{\coker}{\textup{coker\,}}
\newcommand{\Coh}{\textup{Coh}}
\newcommand{\depth}{\textup{depth\,}}
\newcommand{\diag}{\textup{diag}}
\newcommand{\Ext}{\textup{Ext}}

\newcommand{\Hilb}{\textup{Hilb\,}}
\newcommand{\Hom}{\textup{Hom\,}}

\newcommand{\im}{\textup{im\,}}
\newcommand{\pdim}{\textup{pd\,}}
\newcommand{\Pf}{\textup{Pf\,}}
\newcommand{\Pic}{\textup{Pic}}
\newcommand{\PProj}{\textup{Proj}}
\newcommand{\rk}{\textup{rk\,}}
\newcommand{\Sing}{\textup{Sing\,}}
\newcommand{\Supp}{\textup{Supp\,}}

\newcommand{\mult}{\textup{mult\,}}
\newcommand{\Sym}{\textup{Sym\,}}

\newcommand{\bE}{\textsc{E}}

\begin{document}

	\title{\textbf{Pfaffian representations of cubic threefolds}}
	\author{Gaia Comaschi}
	\date{}
\maketitle

\begin{abstract}
	\noindent
Given a cubic hypersurface $X\subset \D{P}^4$, we study the existence of Pfaffian representations of $X$, namely of $6\times 6$ skew-symmetric matrices of linear forms $M$ such that $X$ is defined by the equation $\Pf(M)=0$. It was known that such a matrix always exists whenever $X$ is smooth. Here we prove that the same holds whenever $X$ is singular, hence that every cubic threefold is Pfaffian.
\end{abstract}
\section*{Introduction}
Let $X\subset \D{P}^n$ be a degree $d$ hypersurface defined by a polynomial $F\in \D{C}[X_0,\dots, X_n]$.
It is a classical problem in algebraic geometry to determine whether it is possible to write a power $F^r$ of $F$ 
as the determinant of a matrix $M$ of homogeneous forms. 
For $r=1$, such a matrix $M$ is usually referred to as a \textit{determinantal representation} of $X$; if moreover $M$ has linear entries it is called a \textit{linear} determinantal representation of $X$. 


The existence of determinantal representations of cubic hypersurfaces, in particular of dimension less than or equal to 2, 
is an issue that had already been approached in the 18th century in works by Grassmann (\cite{Grass}). 
It is known that plane cubics and cubic surfaces generally admit linear determinantal equations, whilst they always admit \textit{Pfaffian representations},
that is to say, they can always  be expressed as the zero locus of the Pfaffian of a skew-symmetric matrix of linear forms
(see eg. \cite{Beauville}, \cite{Fabbio}, \cite{CAG}, \cite{FM}).

The first results concerning the Pfaffian representability of  cubic threefolds were obtained by Adler and Ramanan \cite{AR} who proved that a general cubic threefold is Pfaffian. Later Beauville (\cite{B1},\cite{Beauville}), applying results obtained by Markushevich and Tikhomirov \cite{Dima-Tikho}, showed that the same holds for every smooth cubic threefold. In loc. cit. Markushevich and Tikhomirov studied indeed \textit{instanton bundles of minimal charge}, that are stable rank 2 vector bundles with Chern classes $c_1=0, \ c_2=2$, on a smooth cubic threefold $X$. 
Beauville observed then that twisting an instanton $\F$ we get a rank 2 (locally free) skew-symmetric Ulrich sheaf, an ACM sheaf whose minimal free resolution in $\D{P}^4$ is a ``linear'' complex of length one of the form:
$$0\longrightarrow \mathcal{O}^{\oplus 6}_{\D{P}^4}(-1)\xrightarrow{\varphi_1} \mathcal{O}^{\oplus 6}_{\D{P}^4}\longrightarrow \F(1)\longrightarrow 0;$$
accordingly the map $\varphi_1$ (the only \DM{non-vanishing} differential in the resolution of $\F(1)$) provides a Pfaffian representation of $X$. 
This article is devoted to the study of Pfaffian representations of \textit{singular} 3-dimensional cubics.
We prove the following:
\begin{thm*}[Theorem \ref{teoremone_Pf}] A cubic threefold $X\subset \D{P}^4$ always admits a Pfaffian representation.
\end{thm*}
To prove the theorem we will distinguish the following families of singular threefolds:
\begin{itemize}
	\item Normal cubic threefolds that present at most double points.
	\item \DM{Non-normal} cubic threefolds.
	\item Cubic threefolds that are cones.
\end{itemize}
Depending on the type of singularities of $X$ we will adopt a different approach 
to show that $X$ is Pfaffian.
Proving that cones and non-normal cubics are Pfaffian reveals to be quite easy; 
things get more interesting when $X$ is singular in codimension at most 2 and presents just double points.
In this case our strategy essentially consists in adapting the techniques used in \cite{Dima-Tikho} to prove the existence of instanton (or equivalently Ulrich) bundles. 
\

Here is the plan of the paper.
In section 1 we introduce some preliminary material about ACM and Ulrich sheaves,  ACM and AG schemes;
we recall in particular how linear determinantal and Pfaffian representations of a hypersurface $X\subset \D{P}^n$ relate to Ulrich sheaves supported on it.

In section 2 we introduce the so called \textit{Serre correspondence}, one of the most efficient machinery to prove the existence of ACM sheaves.
This technique allows indeed to obtain ACM sheaves on a scheme $X$ starting from codimension 2 AG subschemes of $X$. 
We show that when $X$ is a cubic threefold, the ACM sheaf $\E$ constructed from a \textit{normal AG elliptic quintic} $\C\subset X$ (that is a non-degenerate AG curve $\C\subset X$ with Hilbert's polynomial $P_{\C}(m)=5m$ and such that $H^0(\mathcal{O}_{\C})\simeq \D{C}$) applying Serre correspondence is Ulrich. 

The central part of the article is therefore devoted to the study of normal AG elliptic quintics on a normal cubic threefold $X$.
We provide two methods to show the existence of these curves, both based on the ones applied \DM{in \cite{Dima-Tikho}} to prove that a smooth $X$ always contains a nonsingular normal elliptic quintic.
Since in order to apply both methods we need to figure out which kind of cubic surfaces arise as general hyperplane sections of $X$, it is necessary to  
understand in detail the singularities that $X$ might present. This is done with the help of Segre's classification of cubic threefolds \cite{Segre}, a brief resume of its work is presented in section 3. 

The first method is displayed in section 4. It substantially relies on a deformation argument and allows to prove the existence of a \textit{smooth} normal elliptic quintic that moreover does not pass through any of the singular points of $X$. To start with we show that a general hyperplane section $S$ of $X$ contains a smooth quintic elliptic curve $\C_0$ disjoint from the singular locus of $S$. Then we \DM{prove} that \DM{$\C_0$} deforms to a \DM{nondegenerate} \DM{curve} $\C$ that is still contained in $X_{sm}$, the smooth locus of $X$.
The arguments adopted in the smooth case apply, almost unchanged, whenever $X$ has isolated singularities: a general hyperplane section of $X$ is indeed a cubic surface that still has no singular points; the proof of the existence of a degenerate smooth elliptic quintic is thus equivalent. 
When $\Sing(X)$, the singular locus of $X$, has dimension one, a general hyperplane section $S$ of $X$ is a cubic surface with isolated singularities. In \DM{this case} we will consider 
$$\phi: \tilde{S}\to S,$$
a minimal resolution of singularities of $S$ (the surface $\tilde{S}$ is a so called \textit{weak Del Pezzo surface}) and we will prove the existence of a smooth curve $\tilde{\C}_0$ on $\tilde{S}$ such that \mbox{$\C_0:=\phi(\tilde{\C}_0)\subset \tilde{S}$} is a smooth elliptic quintic disjoint from $\DM{\Sing}(S)$.

In section 5 we exhibit the second method, a constructive one, to produce elliptic quintics.  With this approach we obtain directly a \DM{nondegenerate}, but possibly singular, AG elliptic quintic $\C\subset X$.
To apply it we first need to show the existence of a (not necessarily irreducible) rational quartic $C_4\subset X$. Then we will present \DM{a construction of} a cubic scroll $\Sigma$ containing $C_4$ \DM{such} that the curve $\C$, residual to $C_4$ in $\Sigma\cap X$,
is an AG elliptic quintic. 
These results lead us to conclude the following:
\begin{thm*}[Theorem \ref{normal_quintic}] Let $X$ be a normal cubic threefold that is not a cone. Then there exists a normal AG elliptic quintic curve $\C\subset X$.
\end{thm*}
We will finally show that the Ulrich sheaf $\E$ on $X$ constructed by a curve $\C$ obtained applying both our methods is locally free. This will ensure that $\E$ is skew-symmetric so that we can assert the following: 
\begin{thm*}[Theorem \ref{normal_Pfaff}] Let $X$ be a normal cubic threefold that is not a cone. Then $X$ is Pfaffian.
\end{thm*}

In the last section we treat the cases where $X$ is either a cone over a cubic surface or a 3-dimensional cubic whose singular locus has dimension bigger than or equal to 2. As every cubic hypersurface of dimension \DM{two or less} is Pfaffian, we deduce immediately that the same holds whenever $X$ is a cone.
If $X$ is not normal (and \DM{then it is} singular along a linear space of dimension at least 2), it is easy to write explicitly a skew-symmetric matrix of linear forms $M$ such that  $X$ is defined by the equation $\DM{\Pf}(M)=0$. This will conclude the proof of our main result.
\section{Preliminaries}

\subsection{Rudiments from homological algebra: ACM and AG schemes }

We give some preliminary algebraic notions that we will use throughout the entire article.
\begin{defn} Let $(A,\mathfrak{m})$ be a Noetherian local domain. We say that \DM{an $A$-module $M$ is  Cohen-Macaulay (CM for short)} if:
	$$\DM{\depth}(M)=\DM{\dim\,}(M).$$
\end{defn}

\begin{rmk}\label{local_dim_depth} 
	We recall that we can give the following cohomological characterization, that will come in use later, of the depth and the dimension of a module $M$ over $(A,\mathfrak{m})$.
	
	\begin{itemize}
		\item $H^i_{\mathfrak{m}}(M)=0, \ \forall\: i<\DM{\depth}(M), \ \forall\: i>\DM{\dim\,}(M)$.
		\item $H^i_{\mathfrak{m}}(M)\ne 0$ for $i=\DM{\depth}(M)$ and $ i=\DM{\dim\,}(M)$.
	\end{itemize}
	We see then that \DM{$M$ is a CM $A$-module if and only if it has only one} \DM{non-vanishing local} cohomology group, namely $H^i_{\mathfrak{m}}(M), \ i=\DM{\dim\,}(M)$.
\end{rmk}

\begin{defn} Let $A$ be a Noetherian domain and $M$ an $A$-module. We say that $M$ is \DM{Cohen-Macaulay (CM for short)} if for all maximal ideals $\mathfrak{m}\subset A$, $M_{\mathfrak{m}}$ is a \DM{CM} $A_{\mathfrak{m}}$-module. 
\end{defn}

\DM{Let now $R$ be} the ring of polynomials in $n+1$ variables, $R:=\D{C}[X_0,\dots X_n]$; $R$ is a regular ring hence it is CM. Let $M$ be a graded CM $R$-module.
Localizing at \mbox{$\mathfrak{m}_0:=(X_0,\dots, X_n)$} and applying Auslander-Buchsbaum \DM{formula,}
$$ \DM{\pdim}(M_{\mathfrak{m}_0})+\DM{\depth}(M_{\mathfrak{m}_0})=\DM{\depth}(R_{\mathfrak{m}_0}),$$
we \DM{obtain}: 
$$\DM{\pdim}(M_{\mathfrak{m}_0})=\DM{\dim\,}(R_{\mathfrak{m}_0})-\DM{\dim\,}(M_{\mathfrak{m}_0}).$$
Therefore, $M$ admits a graded minimal free resolution \DM{of length $c=\codim(M)$}:
\begin{equation}\label{free_res}
\cdots 0\longrightarrow F_c \xrightarrow{\varphi{c}} F_{c-1}\xrightarrow{\varphi_{c-1}}\cdots \xrightarrow{\varphi_1} F_0\longrightarrow M \longrightarrow 0,
\end{equation}
where each term $F_i$ is of the form $F_i= \bigoplus\limits_j R(a_{i,j})$. 

\begin{rmk} Summing up, we see that if $M$ is a \DM{graded} CM $R$-module, then:
	\begin{itemize}
		\item For every maximal \DM{ideal} $\mathfrak{m}\subset R$, the only \DM{non-vanishing} local cohomology group is $H^{i}_{\mathfrak{m}}(M_{\mathfrak{m}}), \ i=\DM{\dim\,}(M)$ (this is due to remark \ref{local_dim_depth}).
		\item For every integer $a\in \D{Z}$,  $\DM{\Ext}^i(M,R(a))=0$ whenever $i> \DM{\codim} (M)$ (this is due to the fact that a minimal free resolution of $M$ is of the form (\ref{free_res})).
	\end{itemize}
\end{rmk}

\ 
\DM{Let} now $X$ \DM{be} a closed subscheme of $\D{P}^n:=\DM{\PProj}(R)$ and denote by $R(X)$ its homogeneous coordinate ring,
$R(X)=R/I_X$, (\DM{$I_X$ being} the saturated ideal of $X$).

\begin{defn} A closed subscheme $X\subset \D{P}^n$ is called \textit{\DM{arithmetically} Cohen-Macaulay} (ACM for short),  if its homogeneous coordinate ring $R(X)$ is a Cohen-Macaulay ring. 
\end{defn}

\DM{
	\begin{rmk}
		\label{projnormcurve}
		A projectively embedded curve $C\subset \D{P}^n$ is ACM if and only if it is CM and $H^1(\mathcal I_C(k))=0$ for all $k\in \mathbb Z$. This is closely related to the projective normality: $C$ is projectively normal if and only if it is normal and the the restriction maps on global sections $H^0(\mathcal O_{\mathbb P^n}(k))\to H^0(\mathcal O_{C}(k))$ are surjective for all $k$. In particular, if $C$ is smooth, then it is ACM if and only if it is projectively normal.
	\end{rmk}
}

\begin{defn}A closed subscheme $X\subset \D{P}^n$ of codimension $c$ is called \textit{\DM{arithmetically} Gorenstein} (AG for short)  if the following hold:
	\begin{itemize}
		\item $X$ is ACM.
		\item The canonical module of $R(X)$, $K_X:=\DM{\Ext}_S^c(R(X),R)(-n-1)$, is isomorphic to $R(X)(a)$ for some $a\in \D{Z}$.
	\end{itemize}
\end{defn}  

\begin{rmk} 
	\begin{itemize}
		\item The condition that a subscheme $X\subset \D{P}^n$ of codimension $c$ is AG, is equivalent to requiring that $R(X)$ admits a graded minimal free resolution of the form (\ref{free_res}), such that the term $F_c$ has rank 1.
		\item If $X$ is AG, from the isomorphism $K_X\simeq R(X)(a)$ we get that the canonical sheaf $\omega_X=\sheafext^c_S(\mathcal{O}_X, \mathcal{O}_{\D{P}^n})(-n-1)$ is isomorphic to $\mathcal{O}_X(a)$, for some $a \in \D{Z}$.
	\end{itemize}
\end{rmk}

\DM{\subsubsection*{ACM and Ulrich sheaves}}

\begin{defn} A coherent sheaf $\mathcal{E}$ on $\D{P}^n$ is \DM{an} \textit{\DM{arithmetically} Cohen-Macaulay} (ACM for short) sheaf if $E:=\bigoplus_{j\in \D{Z}}H^0(\mathcal{E}(j))$, its module of twisted global \DM{sections}, is a (graded) Cohen-Macaulay module over $\D{C}[X_0,\dots, X_n]$.
\end{defn}

We observe that a coherent sheaf $\mathcal{E}$ on $\D{P}^n$ is ACM if and only if:
\begin{itemize}
	\item $\mathcal{E}_x$ is a Cohen-Macaulay $\mathcal{O}_{\D{P}^n,x}$ module $\forall x \in \D{P}^n$ (that is, $E_{\mathfrak{m}_x}$ is CM for all maximal ideals $\mathfrak{m}_x$ different from the irrelevant ideal $\mathfrak{m}_0:=(X_0,\dots, X_n))$.
	\item $H^i(\mathcal{E}(j))=0$, for $0<i< \DM{\dim\,}(\DM{\Supp}(\mathcal{E}))$, \ for $j\in \D{Z}$. This condition is equivalent to the fact that $E_{\mathfrak{m}_0}$ is a Cohen Macaulay $\D{C}[X_0,\dots, X_n]_{\mathfrak{m}_0}$ module.
\end{itemize} 

If $\E$ is ACM, from a minimal graded resolution of its module of twisted global \DM{sections} $E$ of the form (\ref{free_res}), we get a locally free resolution \DM{of $\E$}
\begin{equation}\label{free_res_sheaf}
\cdots 0\longrightarrow \bigoplus_j \mathcal{O}_{\D{P}^n}(a_{c,j}) \xrightarrow{\varphi_{c}} \bigoplus_j \mathcal{O}_{\D{P}^n}(a_{c-1,j})\xrightarrow{\varphi_{c-1}}\cdots \xrightarrow{\varphi_1} \bigoplus_j \mathcal{O}_{\D{P}^n}(a_{0,j})\longrightarrow \E \longrightarrow 0
\end{equation}
of length $c=\DM{\codim} (\DM{\Supp}(\E))$. Note that every differential can be represented as a matrix whose entries are homogeneous forms on $\D{P}^n$.

\begin{defn}A coherent sheaf $\mathcal{E}$ on $\D{P}^n$ is \DM{an} Ulrich sheaf if:
	\begin{itemize}
		\item $\mathcal{E}$ is ACM.
		\item  $H^{i}(\mathcal{E}(-i))=0$ for $i\ge 1$ and $H^{i}(\mathcal{E}(-i-1))=0$ for $i\le \DM{\dim\,}(\DM{\Supp}(\mathcal{E})).$
	\end{itemize} 
\end{defn}

\begin{rmk}\label{res_Ulrich}
	Let $\E$ be an ACM sheaf and denote by $c$ the codimension of $\DM{\Supp}(\E)$. Writing a locally free resolution of $\E$ as (\ref{free_res_sheaf}), from the definition we get that $\E$ is Ulrich if and only if this resolution is \textit{linear}, \DM{that is, if} it has the following form:
	\begin{equation}\label{free_res_U}
	\cdots 0\longrightarrow \mathcal{O}_{\D{P}^n}^{\oplus{r_c}}(-c) \xrightarrow{\varphi_c} \mathcal{O}_{\D{P}^n}^{\oplus{r_{c-1}}}(-c+1)\xrightarrow{\varphi_{c-1}}\cdots \mathcal{O}_{\D{P}^n}^{\oplus{r_1}}(-1)\xrightarrow{\varphi_1} \mathcal{O}_{\D{P}^n}^{\oplus{r_0}}\longrightarrow \E \longrightarrow 0.
	\end{equation}
\end{rmk}

\DM{\subsection{Determinantal hypersurfaces}}
Let $V$ be a complex vector space of dimension $n+1$, $F\in S^{d}(V^*)$ \DM{a homogeneous} form of degree $d$ on $\D{P}(V)$ and $r$ an integer bigger than or equal to one. Denote by $\mathcal{M}_{rd\times rd}$ the vector space of square matrices of size $rd$.
It is a classical problem in algebraic geometry to determine whether \DM{some power of} $F$ can be written \DM{as the determinant} of a matrix of linear forms, \DM{that is,} whether there \DM{exists} 
$$M \in \mathcal{M}_{rd\times rd}\otimes_{\D{C}} V^* \ \text{such that} \ F^r=\DM{\det}(M).$$   
The existence of such a matrix $M$ is equivalent to the existence of a Ulrich sheaf of rank $r$ on $X$, the hypersurface in $\D{P}(V)$ defined by the equation $\{F=0\}$. 

\

\begin{prop}Let $X$ be a degree $d$ hypersurface in $\D{P}(V)\simeq \D{P}^n$ defined by an equation $F=0, \ F\in S^{d}(V^*)$.
	The following conditions are equivalent:
	\begin{itemize}
		\item There exists $M\in \mathcal{M}_{rd\times rd}\otimes_{\D{C}} V^*$ such that $F^r=\DM{\det}(M).$
		\item There \DM{exists} an Ulrich sheaf $\mathcal{E}$ supported on $X$ and \DM{of rank $r$ as an $\mathcal{O}_X$-module}.  
	\end{itemize}
\end{prop}
\begin{proof}If $F$ is such that $F^r=\det(M)$ for $M \in \mathcal{M}_{rd\times rd}\otimes_{\D{C}}V^*$, 
	the matrix $M$ defines a morphism of vector bundles over $\D{P}^n$ that fits in a short exact sequence:
	\begin{equation}\label{ulrich_hs}
	0\longrightarrow \mathcal{O}_{\D{P}^n}^{\oplus{rd}}(-1)\overset{M}\longrightarrow \mathcal{O}_{\D{P}^n}^{\oplus{rd}} \longrightarrow \E \longrightarrow 0
	\end{equation}
	\DM{where} $\E:= \coker(M)$. $\E$ is a coherent sheaf supported on $X$, and as \mbox{$c_1(\E)=r[dH]=r[X]$}, $\E$ has rank $r$ on $X$. $\E$ is thus a coherent sheaf admitting a minimal free resolution of length $1=\DM{\codim}(\DM{\Supp}(\E))$, hence it is ACM. 
	Since moreover this resolution is linear, from remark \ref{res_Ulrich} we deduce that $\E$ is Ulrich.
	\DM{Vice versa,} let $\E\in \DM{\Coh\,}(X)$ be a rank $r$ Ulrich sheaf; \DM{ then $\E$} admits a linear minimal free resolution of length 1. The only nonvanishing differential of this resolution $\mathcal{O}^{\oplus{rd}}_{\D{P}^n}(-1)\xrightarrow{\varphi_1} \mathcal{O}^{\oplus{rd}}_{\D{P}^n}$ \DM{defines} an element
	$M\in\mathcal{M}_{rd\times rd}\otimes_{\D{C}} V^*$. As $X=\DM{\Supp}(\E)=\{x\in \D{P}^n \mid \DM{\rk}(\phi(x))\le rd \}$, we deduce that $\DM{\det}(M)=F^r$. 
\end{proof}

We now look for \textit{$\epsilon$-symmetric determinantal representations} of $F$, namely we ask if ever we can find an element $M \in \mathcal{M}_{2d\times 2d}\otimes_{\D{C}} V^*$ such that:
\begin{itemize}
	\item $F^2=\DM{\det}(M)$;
	\item $M$ is $\epsilon$-symmetric that is, $ M^T=\epsilon M$, $\epsilon =\pm 1$.
\end{itemize}
$\epsilon$-symmetric determinantal representations are associated to rank 2 Ulrich sheaves $\mathcal{E}$ on $X$ that, additionally, are \textit{$\epsilon$-symmetric}.

We recall briefly how $\epsilon$-symmetric sheaves are defined.
To start with, consider an ACM-sheaf $\mathcal{E}$ on $\D{P}^n$ supported on a hypersurface $X$ and endowed with a morphism of sheaves $\phi: \mathcal{E}\to \sheafhom(\mathcal{E}, \mathcal{O}_X(N))$ for some integer $N$. Applying  the functor $\sheafhom(\ \cdot, \mathcal{O}_X(N))$ we get a morphism $\mathcal{E}^{\vee \vee} \to \mathcal{E}^{\vee}(N)$, composing then with the natural homomorphism $\mathcal{E}\to \mathcal{E}^{\vee \vee}$ we get $\phi^T$, the \textit{transpose} of $\phi$:
$$\phi^T:\mathcal{E}\to \mathcal{E}^{\vee}(N).$$
\begin{defn} An ACM sheaf $\mathcal{E}$ on $\D{P}^n$, endowed with a sheaf morphism \mbox{$\phi:\mathcal{E}\to \mathcal{E}^{\vee}(N)$}, is called $\epsilon$-symmetric if $\phi$ is an isomorphism and $\phi^T=\epsilon \phi \ (\epsilon=\pm 1).$
\end{defn}

\begin{prop}\label{skew_Ulrich}Let $X$ be a degree $d$ hypersurface in $\D{P}(V)\simeq \D{P}^n$ defined by an equation $F=0, \ F\in S^{d}(V^*)$.
	The following conditions are equivalent:
	\begin{itemize}
		\item There exists \DM{an $\epsilon$-symmetric matrix $M\in \mathcal{M}_{2d\times 2d}\otimes_{\D{C}} V^*$} such that \mbox{$F^2=\DM{\det}(M)$.} 
		\item There exists a rank 2 $\epsilon$-symmetric Ulrich sheaf $\mathcal{E}$ on $X$.
	\end{itemize}
\end{prop}
\begin{proof} See \cite{Beauville}, theorem 2.B.
\end{proof}
\section{Ulrich sheaves and AG normal elliptic quintics}

We now study in detail Pfaffian representations of cubic threefolds.
Throughout the rest of the chapter we denote by $V$ a 5-dimensional linear space, by $F\in S^3(V^*)$ a homogeneous form of degree 3 and by $X\subset \D{P}(V)\simeq \D{P}^4$ the cubic hypersurface defined by the equation $\{F=0\}$.
For what has been discussed so far, we can assert that $X$ is Pfaffian whenever we are able to show \DM{that there} exists 
a rank 2 Ulrich sheaf $\mathcal{E}$ \DM{on $X$} endowed with a sheaf isomorphism $\phi:\mathcal{E}\to \sheafhom(\mathcal{E}, \mathcal{O}_X(2))$, such that $\phi^T=-\phi$ (namely a \textit{skew-symmetric Ulrich sheaf}).
One of the possible ways to prove that, more generally, $X$ carries a rank 2 Ulrich sheaf is the so called \textit{Serre correspondence}: a technique that allows to construct rank 2 ACM sheaves on $X$ starting from codimension 2 subschemes of $X$.
\DM{\subsection{Serre correspondence}}

Throughout the section we suppose that all cubic threefolds $X$ we deal with are integral.
We illustrate how Serre correspondence applies to a (not necessarily smooth) cubic threefold $X$. 
In this specific case a codimension 2 AG subscheme of $X$ is an AG curve $\C\subset X$ and since moreover $X$ is AG as well (as it is a hypersurface) curves of such a kind admit the following useful characterization:

\begin{prop}\label{pd_1} Let $\mathcal{C}\subset X$ be an AG curve and let $\mathcal{I}_{\C}$ be the ideal sheaf of $\C$ in $X$. For every point $x \in \mathcal{C}, \ \DM{\pdim}(\mathcal{I}_{\mathcal{C},x})=1$.
\end{prop}
\begin{proof} By Auslander\DM{-}Buchsbaum formula, we have:
	$$\DM{\pdim}(\mathcal{I}_{\mathcal{C},x})=\DM{\depth}(\mathcal{O}_{X,x})-\DM{\depth}(\mathcal{I}_{\mathcal{C},x}).$$
	$X$ is an AG scheme so $\DM{\depth}(\mathcal{O}_{X,x})=\DM{\dim\,}(\mathcal{O}_{X,x})=3.$ In order to determine $\DM{\depth}(\mathcal{I}_{\mathcal{C},x})$ we compute the local cohomology $H^i_{x}(\mathcal{I}_{\mathcal{C},x})$.
	Consider the short exact sequence of sheaves on $\D{P}^4$:
	$$0\longrightarrow \mathcal{I}_{X/\D{P}^4}\longrightarrow\mathcal{I}_{\mathcal{C}/\D{P}^4}\longrightarrow i_*(\mathcal{I}_{\mathcal{C}})\longrightarrow 0$$
	where $\mathcal{I}_{X/\D{P}^4}$ and $\mathcal{I}_{\mathcal{C}/\D{P}^4}$ denote respectively, the ideal sheaves of $X$ and $\mathcal{C}$ in $\D{P}^4$, and $i$ is the inclusion $i:X\hookrightarrow \D{P}^4$.
	Localizing at $x\in \mathcal{C}$ we get a short exact sequence of $\mathcal{O}_{\D{P}^4,x}-$modules: 
	\begin{equation}\label{ideals_loc}
	0\longrightarrow \mathcal{I}_{X/\D{P}^4,x}\longrightarrow\mathcal{I}_{\mathcal{C}/\D{P}^4,x}\longrightarrow (\mathcal{I}_{\mathcal{C},x})\longrightarrow 0.
	\end{equation}
	Since $H^0_{x}(\mathcal{O}_{\D{P}^4,x})=0$, we deduce that $H^0_{x}(\mathcal{I}_{X/\D{P}^4,x})=0$ and  $H^0_{x}(\mathcal{I}_{\mathcal{C}/\D{P}^4,x})=0$.
	Now, $H^i_{x}(\mathcal{I}_{X/\D{P}^4,x})\simeq H^{i-1}_{x}(\mathcal{O}_{X,x})=0$, for $i=1,\ldots, 3$ (as $\mathcal{O}_{X,x}$ is a local 3-dimensional CM ring).
	Similarly, $H^1_{x}(\mathcal{I}_{\mathcal{C}/\D{P}^4,x})\simeq H^0_{x}(\mathcal{O}_{\mathcal{C},x})=0$ (as $\mathcal{O}_{\mathcal{C},x}$ is a CM local ring of dimension 1).
	From these computations, we see that once we take the long exact sequence in local cohomology from (\ref{ideals_loc})\DM{,}
	
	\[
	\begin{tikzcd}[arrows=to]
	\cdots \rar & H^i_x(\mathcal{I}_{\C/\D{P}^4,x}) \rar &  H^i_x(\mathcal{I}_{\C,x}) \rar &  H^{i+1}_x(\mathcal{I}_{X/\D{P}^4,x}) \rar & \cdots\\
	\end{tikzcd}\]
	
	we get that: 
	\begin{itemize}
		\item  $H^0_{x}(\mathcal{I}_{\mathcal{C},x})=H^1_{x}(\mathcal{I}_{\mathcal{C},x})=0$\DM{,}
		\item $H^2_{x}(\mathcal{I}_{\mathcal{C},x})\simeq H^2_{x}(\mathcal{I}_{\mathcal{C}/\D{P}^4,x})\simeq H^1_{x}(\mathcal{O}_{\mathcal{C},x})\ne 0. $
	\end{itemize}
	These equalities \DM{imply} that $\DM{\depth}(\mathcal{I}_{\mathcal{C},x})=2$, hence $\DM{\pdim}(\mathcal{I}_{\mathcal{C},x})=1$.
	
\end{proof}
\vspace{1mm}

Let's consider now \DM{ an AG curve $\mathcal{C}\subset X$.} The canonical sheaf of $\mathcal{C}$ is isomorphic to $\mathcal{O}_{\mathcal{C}}(a)$ for some integer $a \in \D{Z}$.
From the isomorphisms $$\omega_{\mathcal{C}}\simeq \sheafext^2(\mathcal{O}_C, \omega_X) \simeq \sheafext^2(\mathcal{O}_C, \mathcal{O}_X(-2))\simeq \sheafext^1(\mathcal{I}_{\mathcal{C}}, \mathcal{O}_X(-2))$$ we \DM{obtain}:
$$ \mathcal{O}_{C}\simeq \sheafext^1 (\mathcal{I}_{C}, \mathcal{O}_X(-2-a)).$$ 

As $\mathcal{C}$ is a codimension 2 AG subscheme of $X$, \mbox{$\sheafext^i(\mathcal{O}_{\mathcal{C}}, \mathcal{O}_X)\simeq \sheafext^{i-1}(\mathcal{I}_{\mathcal{C}},\mathcal{O}_X)=0$}, \DM{whenever} $i\ne 2$; moreover, as $\sheafext^2(\mathcal{I}_{\mathcal{C}}, \mathcal{O}_X)$ is supported on $\mathcal{C}$, $H^i(X, \sheafext^1(\mathcal{I}_{\mathcal{C}}, \mathcal{O}_X))=0$ whenever $i\ne 0$ or $i\ne 1$. 
Therefore, applying the local to global spectral sequence (cf.\cite{Godement}. II Th.7.3.3), 
$$ E^{p,q}_2= H^p(X, \sheafext^q( \mathcal{O}_{\mathcal{C}}, \mathcal{O}_X(-2-a)))\Longrightarrow \DM{\Ext}^{p+q}(\mathcal{O}_{\mathcal{C}}, \mathcal{O}_X(-2-a))$$
we get isomorphisms: $$H^0(\mathcal{O}_{\mathcal{C}})\simeq H^0(X,\sheafext^1 (\mathcal{I}_{C}, \mathcal{O}_X(-2-a)))\simeq \DM{\Ext}^1( \mathcal{I}_{\mathcal{C}}, \mathcal{O}_X(-2-a)).$$  
The unit element $1 \in \DM{\Ext}^1( \mathcal{I}_{\mathcal{C}}, \mathcal{O}_X(-2-a))$ corresponds to a short exact sequence
\begin{equation}\label{ACM_Serre}
0\longrightarrow \mathcal{O}_X(-2-a)\longrightarrow \mathcal{N}\longrightarrow \mathcal{I}_{\mathcal{C}}\longrightarrow 0.
\end{equation}
\begin{prop} $\mathcal{N}$ is a rank 2 ACM sheaf.
\end{prop}
\begin{proof}From (\ref{ACM_Serre}) we compute immediately that $\mathcal{N}$ has rank 2 and that the following hold:  
	\begin{itemize}
		\item $H^i_{x}(\mathcal{N}_x)=0$ for $i=0,1, \ \forall\: x \in X$. This is due to the fact that $H^i_x(\mathcal{O}_{X,x})=0$, and $H^i_x(\mathcal{I}_{\mathcal{C},x})=0$ for $i=0,1.$ 
		\item $H^1(\mathcal{N}(j))=0, \ \forall j\in \D{Z}$. This is due to the fact that $\mathcal{C}$ is AG, hence $H^1_*(\mathcal{I}_{\mathcal{C}})=0.$
	\end{itemize}
	To conclude that $\mathcal{N}$ is ACM we still need to check that $H^2_{x}(\mathcal{N}_x)=0$ $\forall \ x \in X$ (so that $\mathcal{N}_x$ is Cohen-Macaulay $\forall x\in X$) and that $H^2(\mathcal{N}(j))=0, \ \forall j\in \D{Z}$.
	It is proved in \cite{Hartshorne_Rao} 1.11,  that the sheaf $\mathcal{N}$ obtained from an extension corresponding to $1 \in  \DM{\Ext}^1( \mathcal{I}_{\mathcal{C}}, \mathcal{O}_X(-2-a))$ satisfies:
	\begin{align*}
	\sheafext^1(\mathcal{N}, \mathcal{O}_X)=&0.\\
	H^1_*(\mathcal{N}^{\vee})=&0.
	\end{align*}
	This allow us to prove that $\mathcal{N}$ is ACM since:
	\begin{itemize}
		\item Applying local \DM{duality,} (cf. \cite{hart-loc}, ch. 6)  as $\mathcal{O}_{X,x}$ is  a local Gorenstein ring, we get: 
		$$H^2_{x}(\mathcal{N}_x)\simeq \Hom({\Ext}^1(\mathcal{N}_x, \mathcal{O}_{X,x}), \omega_{X,x}).$$ 
		This terms are thus both equal to 0 since  $\DM{\Ext}^1(\mathcal{N}_x, \mathcal{O}_{X,x})\simeq \sheafext^1 (\mathcal{N}, \mathcal{O}_X)_x$ and \mbox{$\sheafext^1 (\mathcal{N}, \mathcal{O}_X)=0$.}
		This implies that $\N$ is locally Cohen-Macaulay.
		\item From the previous point $\N$ is locally Cohen-Macaulay. Again, applying local duality we then deduce that $Ext^i_{\mathcal{O}_{X,x}}(\N_x,\mathcal{O}_{X,x})\simeq Ext^i_{\mathcal{O}_{X,x}}(\N_x,\omega_{X,x})=0$, \mbox{$ \forall\: x\in X$}, $ \forall\: i>0$, implying that $\sheafext^i(\N,\omega_X)=0, \ \forall \: i>0$. From the local to global spectral sequence:
		$$E^{p,q}_2=H^p(X,\sheafext^q(\N,\omega_X))\Longrightarrow Ext^{p+q}(\N,\omega_X),$$
		
		we obtain an isomorphism $Ext^i(\N,\omega_X)\simeq H^{i}(X,\sheafhom(\N,\omega_X))$. 
		By Serre's duality we get $Ext^i(\N,\omega_X)\simeq H^{n-i}(\N)^*$ and since $X$ is AG with canonical sheaf $\omega_X=\mathcal{O}_X(-2)$,
		we conclude that $H^{i}(\N^{\vee}(-2))$ is dual to $H^{n-i}(\N)$ for all $i>0$.
		Applying this argument to any twist of $\N$ we get that, for all $j \in \D{Z}$ and for $i=1$, $H^1(\N^{\vee}(-j-2))$ is dual to $H^2(\mathcal{N}(j))$. Since $H^1_*(\mathcal{N}^{\vee})=0$ we conclude that $H^2_*(\mathcal{N})=0$.
		
	\end{itemize}
\end{proof}

\DM{\subsection{Rank 2 Ulrich sheaves from normal AG elliptic quintics}}

Suppose now that $\mathcal{C}\subset X$ is an AG curve satisfying the following conditions:
\begin{itemize}
	\item $\mathcal{C}$ has Hilbert polynomial $P_{\C}(n)=5n$ ( so that $\C$ has arithmetic genus $p_a(\C)=1$, hence $\omega_{\mathcal{C}}\simeq \mathcal{O}_{\mathcal{C}})$;
	\item $H^0(\mathcal{O}_{\C})\simeq \D{C}$;
	\item $\C$ is non-degenerate in $\D{P}^4$ (namely $\C$ spans the entire $\D{P}^4$).
	
\end{itemize} 
From now on we will refer to curves satisfying these properties as \textit{AG normal elliptic quintics}.	
Applying Serre's construction to $\mathcal{C}$ get the short exact sequence (\ref{ACM_Serre}) corresponding to the class of $1 \in H^0(\mathcal{O}_{\mathcal{C}})\simeq \DM{\Ext}^1 (\mathcal{I}_{\mathcal{C}}, \mathcal{O}_X(-2))$ and consequently a rank 2 ACM sheaf $\mathcal{N}$ on $X$. If now we twist (\ref{ACM_Serre}) by $\mathcal{O}_X(2)$, we get:
\begin{equation}\label{twisted}
0 \longrightarrow \mathcal{O}_X\longrightarrow \mathcal{E}\longrightarrow \mathcal{I}_{\mathcal{C}}(2)\longrightarrow 0
\end{equation}
where $\mathcal{E}:= \mathcal{N}(2)$. 
From it, since \mbox{$h^0(\mathcal{I}_{\mathcal{C}}(2))=h^0(\mathcal{O}_X(2))- h^0(\mathcal{O}_{\mathcal{C}}(2))= 15-10=5$}, we compute that $h^0(\mathcal{E})=6$. 
Moreover, from the assumption of non-degeneracy of $\C$ we have $h^0(\mathcal{I}_{\mathcal{C}}(1))=0$ so that $h^0(\mathcal{E}(-1))=0$.

From these fact and as we are assuming that $X$ is integral, we can conclude that $\mathcal{E}$ is Ulrich, due to:
\begin{prop}\label{nondeg_Ulrich}Let $X\subset \D{P}^n$ be an integral hypersurface of degree $d$, and $\mathcal{E}$ an ACM sheaf of rank $r$ such that $h^0(\mathcal{E}(-1))=0$ and $h^0(\mathcal{E})=rd$. Then $\mathcal{E}$ is Ulrich.
\end{prop}
\begin{proof} See \cite{stable_Ulrich}, Lemma 2.2. 
\end{proof}

From what we have explained until now, whenever we are able to prove that an integral cubic threefold $X$ contains an AG elliptic quintic $\mathcal{C}$ spanning the entire projective space $\D{P}^4$, we get the existence of a rank 2 Ulrich sheaf $\mathcal{E}$ on $X$.

\DM{Now we want to locate} the singular points of $\E$, namely those points $x\in X$ where $\E_x$ (or equivalently $\mathcal{N}_x$) is not a free $\mathcal{O}_{X,x}$-module, as follows.
We start by localizing (\ref{ACM_Serre}) at a point $x\in X$, getting in this way a short exact sequence of $\mathcal{O}_{X,x}$ modules:
\begin{equation}\label{Serre_local}
0 \longrightarrow \mathcal{O}_{X,x}\longrightarrow \mathcal{N}_x\longrightarrow \mathcal{I}_{\mathcal{C},x}\longrightarrow 0.
\end{equation}
Note that this short exact sequence corresponds to the class $1_x \in \DM{\Ext}_{\mathcal{O}_{X,x}}^1(\mathcal{I}_{\mathcal{C},x}, \mathcal{O}_{X,x})$.
It's clear that $\forall\: x \notin \C$, $\mathcal{N}_x$ is free. Indeed if $x\notin \C$, we have $\mathcal{I}_{\C,x}\simeq \mathcal{O}_{X,x}$, implying that \ref{Serre_local} splits and that thus $\mathcal{N}_x\simeq \mathcal{O}^{\oplus 2}_{X,x}$.
Hence if $x$ is a singular point of $\E$, we must have $x \in \C$. 
Since by proposition \ref{pd_1}, $\mathcal{I}_{\C, x}$ is a $\mathcal{O}_{X,x}$-module of projective dimension 1, we can check if $x$ is a singular point of $\E$ applying the following proposition, proved by Serre (see \cite{Okonek}, Lemma 5.1.2 ):

\begin{prop}[(Serre)]\label{Serre} Let $A$ be a Noetherian local ring, $I\subset A$ an ideal admitting a free resolution of length 1:
	$$0\longrightarrow A^p\longrightarrow A^q \longrightarrow I\longrightarrow 0.$$
	Let $e \in \DM{\Ext}^1(I,A)$ be represented by the extension
	$$0\longrightarrow A\longrightarrow M\longrightarrow I\longrightarrow 0.$$
	Then $M$ is a free $A$-module if and only if $e$ generates the $A$-module $\DM{\Ext}^1(I,A)$.
\end{prop}

\begin{rmk}\label{Ulrich_lf} We observe that if ever $\E$ is locally free, then it is a skew-symmetric Ulrich bundle of rank 2. 
	This is due to the fact that if $\mathcal{E}$ is a vector bundle, it is endowed with a ``natural'' skew-symmetric form $\phi$, $\phi: \mathcal{E}\to \sheafhom(\mathcal{E}, \mathcal{O}_X(2))$. $\phi$ is just the natural isomorphism:
	$$\phi: \mathcal{E} \xrightarrow{\simeq} \mathcal{E}^{\vee} \otimes \bigwedge ^2 \mathcal{E}$$
	that on each fibre $x\in X$ acts as $\phi_x: \mathcal{E}_x \xrightarrow{\simeq} \mathcal{E}_x^{\vee} \otimes \bigwedge ^2 \mathcal{E}_x$, \DM{$ s\mapsto(t\mapsto s\wedge t)$.}
	Furthermore, whenever $\mathcal{E}$ is locally free, from (\ref{twisted}) we can compute that $\DM{\deg\,}(\mathcal{E})=2$, getting, finally, $$\phi: \mathcal{E} \xrightarrow{\simeq} \mathcal{E}^{\vee} \otimes \bigwedge ^2 \mathcal{E}\simeq \sheafhom(\mathcal{E}, \bigwedge^2 \mathcal{E})\simeq \sheafhom(\mathcal{E}, \mathcal{O}_X(2)).$$
	From this remark we understand that the existence of a rank 2 Ulrich bundle on $X$ is sufficient to conclude that $X$ is Pfaffian.
	In the next section we will show how to obtain, on certain singular cubic threefolds, rank 2 Ulrich sheaves that are, for construction, locally free.
	Anyway note that conversely, the existence of a Pfaffian representation of a cubic $X$ does not ensure the existence of a Ulrich \textit{bundle} on $X$; we can indeed have Pfaffian representations corresponding to skew-symmetric Ulrich sheaves that are not necessarily locally free. Anyway in this article we won't deal with these cases.
\end{rmk}

\subsubsection*{The smooth case}Whenever $X$ is smooth, it was proved in \cite{Dima-Tikho}, \cite{B1}, that there always exists a smooth normal elliptic quintic and that the sheaf bulit from it by Serre correspondence is a rank 2 Ulrich bundle.
There are essentially two different ways to prove the existence of normal elliptic quintics on $X$. 
The first is to prove that there exists a smooth elliptic quintic contained in a hyperplane section and then show that it deforms to a \DM{nondegenerate} one, \DM{in the sense that it spans the whole projective space $\mathbb P^4$.}
The second method is ``constructive'' and \DM{produces} directly a \DM{nondegenerate} elliptic quintic. The curve is obtained as the residual \DM{one} \DM{to} a rational \DM{normal} quartic $\C_4 \subset X\cap S$  \DM{in $X\cap \Sigma$ , where $\Sigma$ is a cubic scroll.}
We will later describe in greater \DM{detail} these techniques showing how we can adapt them to singular cases.

\

Next sections are devoted to the study of normal AG elliptic quintics on normal cubic threefolds that present at most double points.
Adapting the two methods used in the smooth case we prove the following:
\begin{thm}\label{normal_quintic} Let $X$ be a normal cubic threefold that is not a cone. Then there exists a normal AG elliptic quintic curve $\C\subset X$.
\end{thm}
We will then show that a normal AG elliptic quintic $\C\subset X$ obtained from both our methods yields, applying Serre correspondence, a rank 2 Ulrich sheaf $\E$ on $X$ that furthermore is locally free:
\begin{thm}\label{Ulrich-normal} Let $X$ be a normal cubic threefold that does not contain triple points. Then there always exists a rank 2 Ulrich bundle $\E$ on $X$.
\end{thm}
From this result we deduce immediately the Pfaffian representability of $X$ since, as we saw in remark \ref{Ulrich_lf},  a rank 2 Ulrich bundle is always skew-symmetric.
Before illustrating the proofs of theorems \ref{Ulrich-normal} and \ref{normal_quintic} we present, following Segre's work \cite{Segre}, a brief essay on normal cubic threefolds.

\section{Normal cubic threefolds}

Let $X$ be a normal cubic threefold that does not contain triple points.
In this section we will describe the singularities that $X$ might present; before doing so we recall some generalities about singular hypersurfaces.

If $X$ \DM{presents} a singular point $x\in X$, we might then choose homogeneous coordinates $X_0, \ldots, X_4$ on $\D{P}^4$ in such a way that $x$ is the point $[1:0:0:0:0]$; consequently the polynomial $F$ that defines $X$ can be written as:
$$F(X_0,\dots, X_4)= X_0 q_x(X_1,\dots, X_4)+ c_x(X_1,\dots, X_4).$$
The locus $Q_x$ defined by the equation $q_x(X_1,\dots X_4)=0$ is a quadric hypersurface (singular in $x$) called the \textit{quadric tangent cone} of $X$ at $x$. It can be characterized as the set of points  \mbox{$Q_x=\{y \in \D{P}^4 \mid \ mult_{x}(F\restriction_{\overline{xy}})\ge 3\}.$}
\begin{defn}We say that the point $x\in X$ is a double point of $r$-th type if $Q_x$, the quadric tangent cone to $X$ at $x$ is a quadric of rank $5-r$. 
\end{defn}
In the upcoming sections we will also refer to a double point $x\in X$ of \DM{$r$-th} type, for $r=3,2,1$ as, respectively, a conic node, a binode and a unode.

Take now a hyperplane $H_x\subset \D{P}^4$ not passing through $x$. Up to a suitable change of coordinates we can assume that $H_x$ has equation $X_0=0$. We consider the locus $\mathcal{D}^{2,3}_x\subset H_x$ defined as:
$$ \mathcal{D}^{2,3}_x:= H_x\cap Q_x \cap X.$$
This is a complete intersection \DM{sextic} curve and we see that for every point $y \in \mathcal{D}^{2,3}_x$, the polynomial $F\restriction_{\overline{xy}}$ has a root of multiplicity at least 3 in $x$ and a root in $y$. Since $\DM{\deg\,}(F)=3$, we deduce that the entire line $\overline{xy}$ is contained in $X$. 
This means that the union of lines passing through $x$ and contained in $X$ is the cone of vertex $x$ over \DM{$\mathcal{D}^{2,3}_x$}. 
Note also that $Q_x\cap H_x$ and $X\cap H_x$ meet at a point $y \in \mathcal{D}^{2,3}_x$ with multiplicity bigger then one if and only if the line $\overline{xy}\subset X$ passes through a singular point of $X$ different from $x$.

Supposing that $X$ is normal, we have that \DM{the} singular locus $\DM{\Sing}(X)$, \DM{has} dimension at most 1. 
If ever $X$ has non-isolated singularities, $\DM{\Sing}(X)$ \DM{contains} then a curve $Y$. Consequently we give the following definition: 
\begin{defn} We say that a curve $Y\subset \DM{\Sing}(X)$ is a \textit{double curve of r-th type} if every point in $Y$ is a double point and the generic point of each of its irreducible components is a double point of \DM{$(r-1)$-th} type.
\end{defn}
Under the additional assumption that $X$ is not a cone (namely that $X$ does not contain triple points), following Segre's classification of \DM{three-dimensional} cubic hypersurfaces \DM{(\cite{Segre})}, we \DM{conclude} that
$\DM{\Sing}(X)$ might \DM{consist} of:
\begin{itemize} 
	\item $N\le 10$ isolated singular points if $\DM{\dim\,}(\DM{\Sing}(X))=0$ (\DM{including the case when} $\DM{\Sing}(X)=\emptyset$).
	\item The union $Y\cup Z$ of a finite length scheme $Z$ (including the case when $Z=\emptyset$) and a \DM{double curve} $Y$, \DM{if} $\DM{\dim\,}(\DM{\Sing}(X))=1$, \DM{where $Y$ is one of the following:}
	\begin{itemize}
		\item A line of first or second type.
		\item A conic of first or second type (possibly degenerating \DM{into a} union of two incident lines).
		\item The union of three non-coplanar lines of first type meeting at a point.
		\item A rational quartic curve of first type (that can possibly degenerate in the union of two conics meeting at a point but not both contained in a hyperplane).
	\end{itemize}
\end{itemize}

\begin{rmk}\label{forma-poly}
	Note that saying that $X$ has multiplicity 2 along the curve $Y\subset X$, means that all the partial derivatives of order 1 of $F$ vanish along $Y$. In other words, denoting by $I_Y\subset \D{C}[X_0,\dots,X_4]$ the homogeneous ideal defining $Y$, we have that
	for $i=0,\dots, 4$, $\frac{\partial F }{\partial X_i} \in I_Y$. This implies in particular that the polynomial $F$ can be written as \mbox{$F=\sum_{i=0}^4 X_i Q_i$} with $Q_i\in {I_Y}_2:= (I_Y\cap \D{C}[X_0,\dots X_4]_2)$ (the space of homogeneous forms of degree 2 belonging to $I_Y$). The hypothesis that $X$ has no triple points ensures that $\forall\: y \: \in Y$, there exists at least one partial derivative of $F$ of order 2 not vanishing at $y$.   
\end{rmk}

\

In the upcoming sections we will need to study the behavior of (general) hyperplane sections of $X$ and to this aim it is necessary to understand the kind of singularities that they might present. If ever $X$ has isolated singularities, applying Bertini's theorem we have that for a general hyperplane $H\subset \D{P}^4$, 
$S:=H\cap X$ is smooth. Whenever $\dim(\Sing(X))=1$, a generic hyperplane section of $X$ is a cubic surface with isolated singularities.
More precisely, taking a hyperplane $H\subset \D{P}^4$ corresponding to a point $h \in {\D{P}^4}^*$ that belongs to the open ${\D{P}^4}^*\setminus (X^*)$ (here $X^*$ is the dual variety of $X$), the intersection $H\cap X$ will be singular along $\DM{\Sing}(X)\cap H$.  
When $X$ is a cubic threefold in the aforementioned list, following Segre's classification, we have at most a finite number of singular points of $X$ not belonging to the curve $Y$.
This means that for a general hyperplane $H\subset \D{P}^4$, the intersection $S:=H\cap X$ satisfies $\DM{\Sing}(S)=Y\cap H$, therefore $S$ is a cubic surface with at most 4 double points.
In order to define the type of these singularities it is necessary to describe in greater detail the quadric tangent cones to $X$ at points belonging to 1-dimensional components of $\Sing(X)$. This is done writing down a normal form for the polynomial $F$ and with the aid of Segre's work. Our study of the singularities of $S$ relies on Bruce and Walls classification of cubic surfaces presented in \cite{Bruce_Wall}.

\subsection{Cubic threefolds singular along a line}

We suppose here that $X$ contains a double line $Y$ supported on the line $\overline{Y}=\nu_1(\D{P}^1)$ where $\nu_1$ is defined as:

\begin{align*}
\nu_1: \D{P}^1 \longrightarrow& \D{P}^4\\
[t_0:t_1]\mapsto& [0:0:0:t_0:t_1]. 
\end{align*}
For an appropriate choice of homogeneous coordinates $X_0, \ldots, X_4$ we can assume that $\overline{Y}$ has equations $\{X_0=X_1=X_2=0\}$. 
We call $U:=\D{C}\langle X_0,\dots X_4\rangle, \ U':=\D{C}\langle X_0,\: X_1,\: X_2\rangle,$ \mbox{$ U'':=\D{C}\langle X_3,X_4\rangle$.}
We get a decomposition $$ S^3(U)=\bigoplus_{i=0}^3 S^i(U')\otimes S^{3-i}(U'')$$ and consequently $F$ can be written as 
$$F=\sum _{i=0}^3 F_i, \hspace{3mm} F_i \in S^i(U')\otimes S^{3-i}(U'')$$
Because of the fact that $X$ has multiplicity 2 along $\overline{Y}$, we have $F_0=0, \ F_1=0$ hence (since moreover we are supposing that $X$ has no triple points) $F$ can be reduced to the following form:

$$F(X_0,\dots X_4)=X_3 q_3(X_0,X_1,X_2)+X_4q_4(X_0,X_1,X_2) + c(X_0,X_1,X_2). $$

with $\DM{\deg\, (}q_3(X_0,X_1,X_2))=\DM{\deg\, (}q_4(X_0,X_1,X_2))=2$ and $\DM{\deg\, (}c(X_0,X_1,X_2))=3$.  

The quadric tangent cones to points in $Y$ draw a \DM{pencil} of quadrics $\mathcal{Q}_{Y}\subset |\mathcal{O}_{\D{P}^4}(2)|$. The quadric tangent cone to the point $[t_0:t_1]\in \overline{Y}$ is the element $q_{[t_0:t_1]}\in \mathcal{Q}_Y$ defined by:

\begin{equation}\label{pencil_qc}
q_{[t_0:t_1]}= t_0 q_3(X_0,X_1,X_2) + t_1 q_4(X_0,X_1,X_2).
\end{equation}
We observe that every element of $\mathcal{Q}_Y$ is singular along the line $\overline{Y}$, hence it has rank less \DM{than} or equal to 3. 
We recognize here two subcases:
\begin{itemize}
	\item A general element in $\mathcal{Q}_Y$ has rank 3 (namely the line $Y$ is of \textit{first type});
	\item Every quadric in $\mathcal{Q}_Y$ has rank at most 2 (namely the line $Y$ is of \textit {second type}).
\end{itemize}

\

\textbf{$Y$ is a line of first type}

If $Y$ is a line of first type, we see from (\ref{pencil_qc}), that on $Y$ we have 3 binodes $x_1, \ x_2, \ x_3$ (these correspond to the three points where the \DM{pencil} $\mathcal{Q}_Y$ meets the locus of quadrics of rank \DM{less than} 3). 
\begin{prop} Let $X$ be a cubic threefold singular along a line $Y$ of first type. Then a general hyperplane section of $X$ is a cubic surface with one $A_1$ singularity.
\end{prop}
\begin{proof}
	We know that for $H$ general, the cubic surface $S:=H\cap X$ is singular along $Y\cap H$. Whenever such a hyperplane $H$ meets the line $Y$ in a point $y_0\ne x_i$ for $i=0,1,2$ (again, all these are ``open conditions'' on ${\D{P}^4}^*$), 
	the resulting intersection $S$ is a cubic surface whose only singular point is $y_0$. The condition that $Y\cap H=\{y_0\}$ implies that $Y\not\subset H$, so that the quadric tangent cone to $S$ at $y_0$ is defined by a quadric of rank 3 (singular just in $y_0$), hence $y_0$ is a conic node.
	$S$ has thus an $A_1$ singularity in $y_0$. 
\end{proof}

\textbf{$Y$ is a line of second type}

If the line $Y$ is of second type, a general element $y_0\in Y$ is a binode; every quadric in the \DM{pencil} $\mathcal{Q}_Y$ is singular along a plane \DM{and is a pair of linear spaces $\mathbb P^3$}.
For a suitable change of coordinates we might suppose that $y_0=[0:0:0:0:1]$, so that $F$ can be written as:
$$F=X_4 q_4(X_0,X_1) + c'(X_0,X_1,X_2,X_3).$$
Consider $H_{y_0}\subset \D{P}^4$ a hyperplane orthogonal to $y_0$, that we can assume having equation $\{X_4=0\}.$ 
All lines passing through $y_0$ are of the form $\overline{y_0x}$, where $x$ varies along the sextic curve $\mathcal{D}^{2,3}_{y_0}\subset H_{y_0}$:
$$\mathcal{D}_{y_0}^{2,3}=\{X_4=q_4(X_0,X_1)=c'(X_0,X_1,X_2,X_3)=0\}.$$
This curve is not irreducible but \DM{is} the union of two plane cubic curves. Indeed, denote by $S_{y_0}$ the intersection $X\cap H_{y_0}$ and by $Q_{y_0}$ the quadric defined by the equations \mbox{$\{X_4=0, \ q_4(X_0,X_1)=0\}$.}
$Q_{y_0}$ is a union of two planes $\Delta_0, \ \Delta_1$, so that $\mathcal{D}_{y_0}^{2,3}$ is equal to $(S_{y_0} \cap \Delta_0)\cup (S_{y_0}\cap \Delta_1)$ .
This also allows us to notice that $X$ can have at most one other singular \DM{point}.
That's because all singular points of $X$ lie on lines $\overline{y_0x}$\DM{,  where $x$ is} a point in $\mathcal{D}_{y_0}^{2,3}$ where \DM{$S_{y_0}$ and $Q_{y_0}$} meet with multiplicity \DM{greater than} one. As $Q_{y_0}=\Delta_0\cup \Delta_1$, we deduce that all singular points must project to $S_{y_0}\cap L_{0,1}$ \DM{where} 
$L_{0,1}\subset H_{y_0}$ \DM{is} the line $\Delta_0\cap \Delta_1$. This intersection consist of the point $Y\cap H_{y_0}$ and of at most one other point (otherwise we would have $X$ singular along the plane $\DM{\langle\,}y_0, L_{0,1}\DM{\,\rangle}$).

\begin{prop}
	Let $X$ be a cubic threefold singular along a line of second type. Then a general hyperplane section of $X$ is a cubic surface with one $A_2$ singularity.
\end{prop}
\begin{proof}
	Consider a hyperplane $H$ such that $\DM{\Sing}(X\cap H)$ consists of $Y\cap H=y_0$ and that moreover satisfies the condition $\rk (q_4\restriction_H)=2$ ($H$ is a general hyperplane through $y_0$). Under generality assumptions, we might also suppose that $H\cap (L_{0,1}\cap S_{y_0})=\emptyset$. The cubic surface $S=X\cap H$ presents just one singular point at $y_0$, a binode. Following \cite{Bruce_Wall} Lemma 3, in order to determine the nature of the singularity, we have to look at the intersection $H\cap (S_{y_0}\cap L_{0,1})$. Since for our choice of $H$ this intersection is empty, we can conclude that the point $y_0$ is an $A_2$ singularity.
\end{proof}

\subsection{Cubic threefolds singular along a conic}

Let be $X$ be a cubic threefold containing a double conic $Y$ supported on $\overline{Y}=\nu_2(\D{P}^1)$ where $\nu_2$ is defined as:

\begin{align*}
\nu_1: \D{P}^1 \longrightarrow& \D{P}^4\\
[t_0:t_1]\mapsto& [t_0^2:t_0t_1:t_1^2:0:0].
\end{align*}
Choose \DM{coordinates} on $\D{P}^4$ in such a way that $\overline{Y}:=\{X_3=0,\  X_4=0,\ q(X_0,X_1,X_2)=0\}$, \DM{where} $q(X_0,X_1,X_2)$ \DM{is} the polynomial $X_1^2-X_0X_2$. We know that if $Y\subset \Sing(X)$, we can then write $F$ as $F=\sum_{i=0}^4 X_i Q_i$ with $Q_i \in (X_3,\ X_4,\ q(X_0,X_1,X_2))_2$.
Hence 
\begin{equation}\label{pol-conic}
F=\sum_{i,j\in \{3,4\}} l_{ij}(X_0,\dots X_4)X_iX_j + l_{012}(X_0,\dots X_4)q(X_0,X_1,X_2)
\end{equation}
where the $l_{ij}$s and $l_{012}$ are linear forms.
From \ref{pol-conic} we get:
$$F=c(X_3,X_4)+ \sum_{i=0}^2 X_i q_i(X_3,X_4) + l(X_3,X_4) q(X_0,X_1,X_2)+ a(X_0,X_1,X_2)$$
where $\deg(l(X_3,X_4))=1, \ \deg q_i(X_3,X_4)=\deg(q(X_0,X_1,X_2))=2, \ i=0,\dots,2$ and $\deg(c(X_3,X_4))=\deg(a(X_0,X_1,X_2))=3.$
Call $\Delta\simeq \D{P}^2$ the plane defined by the equations $\{X_3=0, X_4=0\}$. As $Y\subset \Delta$, the plane cubic $X\cap \Delta:=\{a(X_0,X_1,X_2)=0\}$
must be singular along $Y$ hence we can conclude that $a(X_0,X_1,X_2)$ must be equal to zero.
Finally we obtain that $F$ can be reduced to the following form:
\begin{equation}\label{conic}
F= c(X_3,X_4)+ \sum_{i=0}^2 X_i q_i(X_3,X_4) + l(X_3,X_4) q(X_0,X_1,X_2)
\end{equation} 
This time the quadric tangent cones to $X$ at points on $Y$ describe a conic $\mathcal{Q}_{Y}\subset |\mathcal{O}_{\D{P}^4}(2)|$. The quadric tangent cone at the point $[t_0:t_1]\in \overline{Y}$ is the element $q_{[t_0:t_1]} \in \mathcal{Q}_{Y}$:

$$q_{[t_0:t_1]}= (t_0^2 q_0(X_3,X_4)+t_0t_1 q_1(X_3,X_4)+ t_1^2q_2(X_3,X_4)) - ({t_0}^2X_2-2 t_0t_1 X_1+{t_1}^2X_0) l(X_3,X_4)$$

\

\textbf{$Y$ is a conic of first type}

We suppose now that the generic element of $\mathcal{Q}_{Y}$ is a quadric of rank 3 (namely $Y$ is a double conic of \textit{first type}).
If this is so, the conic $Y$ presents two binodes $x_1, \ x_2$. These are the two points whose coordinates $[t_0:t_1]$ are such that the quadric of equation $ t_0^2 q_0(X_3,X_4)+t_0t_1 q_1(X_3,X_4)+ t_1^2q_2(X_3,X_4)$ 
contains the hyperplane $\{l(X_3,X_4)=0\}.$

\begin{prop}
	Let $X$ be a cubic threefold singular along a conic of first type. Then a general hyperplane section of $X$ is a cubic surface with 2 $A_1$ singularities.
\end{prop}
\begin{proof}
	
	For any point $y\in Y$ $y\ne x_i, i=1,2$, the quadric tangent cone at $y$ is a quadric singular along the line $\D{T}_y\overline{Y}$. 
	Consider now a hyperplane $H\subset \D{P}^4$ such that \mbox{$\DM{\Sing}(X\cap H)=Y\cap H$} consists of two points $y_1, \ y_2$ different from $x_1$ and $x_2$.
	Choosing this hyperplane in such a way that $\D{T}_{y_i}\overline{Y}\not\subset H$, for $i=1, \ 2$; we get that the hyperplane section $S:=X\cap H$ is a cubic surface presenting just two conic nodes $y_1, \ y_2$.
	$S$ is thus a cubic surface with $2 A_1$ singularities.
	
\end{proof}

\begin{rmk}
	The conic $Y$ might also be supported on $\overline{Y}=L_0\cup L_1$, the union of two lines meeting at a point $x_0$. We still can choose coordinates on $\D{P}^4$ in a way that $F$ has the form \ref{conic}, but this time the degree two form $q$ can be written as $q(X_0,X_1)=X_0X_1$.  
	We assume that the line $L_i$ has equations $\{X_3=X_4=X_i=0\}$, for $i=0,1$. 
	The quadric tangent cones at points in $Y$ define now two \DM{pencil}s of quadrics $\mathcal{Q}_{L_i}\subset |\mathcal{O}_{\D{P}^4}(2)|, \ i=1,2$, meeting at a point. Choose an arbitrary point $y$ on one of these lines, say $L_0$, \mbox{$y=[0:t_0:t_1:0:0]$.} 
	The quadric tangent cone at this point is:
	$$q_{L_0,[t_0:t_1]}= (t_0 q_0(X_3,X_4)+t_1 q_2 (X_3,X_4))+ l(X_3,X_4)(t_0X_1).$$
	Hence a general point on each line is a conic node whether the point $x_0$ is either a binode or a unode.
	For a general hyperplane $H$, $X\cap H$ is again a cubic surface having two $A_1$ singularities in $H\cap L_i, \ i=0,1$.
\end{rmk}

\paragraph{$Y$ is a conic of second type}

\

Suppose now that $Y$ is a conic of second type, namely a general point of $Y$ is a binode. 
The condition that $Y$ is a double conic of second type is equivalent to having $l(X_3,X_4)|q_i(X_3,X_4)$ for $i=0,1,2$ in the equation \ref{conic}. Up to a change of coordinates the polynomial $F$ can then be written as:
$$F=c(X_3,X_4) + X_3(\sum_{i=0}^2 \alpha_i (X_4X_i)) + X_3 q(X_0,X_1,X_2)$$
where $c(X_3,X_4)$ is a form of degree 3 such that $X_3 \nmid c(X_3,X_4)$. This time the conic $\mathcal{Q}_{Y}\subset |\mathcal{O}_{\D{P}^4}(2)|$ parametrizing quadric tangent cones to $X$ at points $[t_0:t_1]\in \overline{Y}$ has the form:
$$ q_{[t_0:t_1]}= X_3 [ t_0^2(\alpha_0 X_4-X_2)+ t_0t_1(\alpha_1 X_4-2X_1) + t_1^2(\alpha_2 X_4-X_0)]. $$
Each of these quadric cones decomposes in the union of two spaces. One of them is a hyperplane $T$ fixed ( it is the hyperplane of equation $X_3=0$), the other varies along a family of hyperplanes $\mathcal{H}_{Y,[t_0:t_1]}, \ [t_0:t_1]\in \D{P}^1$ over $Y$. 
We denote by $h_{[t_0:t_1]}$ the linear form defining $\mathcal{H}_{Y,[t_0:t_1]}$.

\begin{prop}
	Let $X$ be a cubic threefold singular along a conic of second type. Then a generic hyperplane section of $X$ is a cubic surface with 2 $A_2$ singularities.
\end{prop}
For a general hyperplane $H \subset \D{P}^4$, the cubic surface $S=X\cap H$ is singular along 2 binodes $\{y_1, \ y_2\}= Y\cap H$.
Arguing exactly as for the case of cubic threefolds containing a double line of second type, we can choose $H$ in such a way that  at each point $y_i$, the surface $S$ has an $A_2$ singularity. 
Indeed taking any point $y$ having coordinates $[t_0:t_1]$ on $Y$ and a plane $H_y$ orthogonal to it defined by a linear equation $h(X_0, \dots,X_4)=0$, we compute that the intersection $H_y\cap X\cap \mathcal{H}_{Y,[t_0:t_1]}\cap T$ is defined by the equations
$$h(X_0,\dots,X_4)=h_{[t_0:t_1]}(X_0,\dots X_4)=X_3=c(X_3,X_4)=0$$
and consists then in at most 3 points. Imposing the additional (open) condition that $H$ does not pass through any of \DM{these points}, we finally have that $S=H\cap X$ presents two $A_2$ singularities at $y_1, \ y_2$.
\begin{rmk} Again, the conic $Y$ might also degenerate to the union of two double lines of second type $Y_1, \ Y_2$ meeting at a point $x_0$.
	If this is the case, every point $y \in Y$ is a binode except from the point $x_0$ that is a unode. Anyway arguing exactly as in the smooth case we can conclude that a general hyperplane section of $X$ is still a cubic surface with 2$A_2$ singularities. 
\end{rmk}

\subsection{Cubic threefolds singular along three non coplanar lines meeting at a point}
Suppose that $X$ contains a double curve $Y\subset \DM{\Sing}(X)$ supported on the union of three lines $L_0, L_1, L_2$  meeting at a point $x_0$ and not lying in the same plane. Choose linear coordinates in such a way that the lines $L_i$s are given by:
$$L_0=\{X_3= X_1=X_2=0\} \hspace{3mm} L_1=\{X_3=X_0=X_2=0\} \hspace{3mm} L_2=\{X_3=X_0=X_1=0\}.$$
For this choice of coordinates on $\D{P}^4$ the homogeneous ideal defining the curve $Y$ is $(X_3,\ X_0X_1,\ X_0X_2, \ X_1X_2)$ and $x_0$ is the point $[0:0:0:0:1]$. 
$F$ belongs to the ideal $(X_3,\ X_0X_1,\ X_0X_2, \ X_1X_2)$ and since $F$ is singular along each line $L_i$ we have $$F=\sum\limits_{i,j \in\{0,1,2 \}}l_{ij}(X_0,\dots X_4) X_iX_j + h(X_0,\dots X_4)X_3^2,$$
where $\deg(l_{ij}(X_0,\dots X_4))=\deg(h(X_0,\dots X_4))=1$ leading to: 
$$F=\alpha X_3^3+ X_3^2 l(X_0,X_1,X_2,X_4) + X_3 q(X_0,X_1,X_2) + X_4(\sum_{i=0}^2 \beta_i \frac{ X_0X_1X_2}{X_i})$$
\DM{where $l$ is} a linear form $l=\sum\limits _{i\ne 3} a_i X_i$ and $q$ \DM{is} a degree 2 polynomial $q=\sum\limits_{\substack{i,j=0 \\ i\ne j}}^{2}b_{ij}X_iX_j.$

The quadric tangent cones at points in $Y$ \DM{determine} three \DM{pencil}s of \sloppy \DM{quadrics \mbox{$\mathcal{Q}_{L_i}\subset |\mathcal{O}_{\D{P}^4}(2)|$}:}
$$\mathcal{Q}_{L_i}=t_0( a_i X_3^2+ X_3( \sum\limits_{\substack{j =0, \\ i\ne j}}^2 b_{ij}X_j) +  X_4( \sum\limits_{\substack{j =0, \\ i\ne j}}^2 \beta_{j}X_j)) + t_1( a_4 X_3^2 + \sum_{j=0}^2 \beta_j \frac{X_0X_1X_2}{X_j}),$$ 
with $[t_0:t_1] \in \D{P}^1, \ i,j,k \in \{0,1,2\}. $

We see thus that a general element of each line is a conic node and that the point $x_0$ is a unode. 
\begin{prop}
	Let $X$ be a cubic threefold singular along three non coplanar lines meeting at a point. Then a general hyperplane section of $X$ is a cubic surface with 3 $A_1$ singularities.
\end{prop}
\begin{proof}
	Arguing exactly as in the previous cases of curves of first type, we see that for a general hyperplane $H$, $S=H\cap X$ is a cubic surface with three conic nodes $y_0,\ y_1, y_2,$ where $y_i:=H\cap L_i$ ; hence a cubic surface with 3 $A_1$ singularities. 
\end{proof}

\subsection{Cubic threefolds singular along a rational normal quartic}
We consider $X$, the secant variety of a rational quartic curve $Y$. The singular locus of $X$ is the entire curve $Y$. 
We express $Y$ as the image of the embedding:
\begin{align*}
\nu_4:\D{P}^1 \longrightarrow & \D{P}^4\\
[t_0:t_1]\mapsto & [t_0^4:t_0^3t_1:t_0^2t_1^2:t_0t_1^3:t_1^4].
\end{align*}
$Y$ is the intersections of the quadrics:
$$ q_0= X_2X_4 -X_3^2, \ q_1=X_2X_3-X_1X_4, \ q_2=2X_1X_3-3X_2^2,$$  $$q_3=X_1X_2-X_0X_3, \ q_4=X_0X_2-X_1^2.$$

defined by the minors of the matrix:

\begin{equation*}
\begin{pmatrix}
X_0 & X_1& X_2 &X_3 \\
X_1 & X_2 & X_3 & X_4
\end{pmatrix}
\end{equation*}
The polynomial $F$ defining $X$, secant variety of $Y$, is the determinant of the matrix:

\begin{equation*}\label{det-ssecant}
N=\begin{pmatrix}
X_0 & X_1 & X_2 \\
X_1 & X_2 & X_3 \\
X_2 & X_3 & X_4 

\end{pmatrix}
\end{equation*}
Therefore we get:
$$F=X_0(X_2X_4-X_3^2) +X_2(X_1X_3 -X_2^2)+ X_1( X_3X_2-X_1X_4).$$

The quadric tangent cones to $X$ at points on $Y$ \DM{belong} to a rational quartic \textit{$\mathcal{Q}_{Y}$} contained in $|\mathcal{O}_{\D{P}^4}(2)|$. The quadric tangent cone at the point $[t_0:t_1]\in Y$ is defined by:
\begin{equation}\label{tc_quartic}
q_{[t_0:t_1]}=t_0^4 q_0+ 2 t_0^3t_1 q_1+ t_0^2t_1^2 q_1+2t_0t_1^3q_3+t_1^4q_4.
\end{equation}
\begin{prop}
	Let $X$ be the secant cubic threefold. Then a general hyperplane section of $X$ is a cubic surface with 4 $A_1$ singularities.
\end{prop}
\begin{proof}
	
	For a general hyperplane $H\subset \D{P}^4$, the cubic surface $S:=H\cap X$ is singular along $Y\cap H$ hence it presents 4 singular points $y_1, \dots y_4$.
	From (\ref{tc_quartic}) we see that a general point of $Y$ is a conic node, therefore under generality assumptions we can suppose that $S$ presents 4 $A_1$ singularities at the points $y_i$.
\end{proof}

\begin{rmk}
	A possible degeneration of the situation we have just studied occurs when the quartic consists of the union of two conics $Y_1, \ Y_2$ meeting in a point $x_0$ but not lying in a same hyperplane. The general point of each conic is a conic node (more precisely every point in $Y$ is a conic node apart from the point $x_0$ that is a binode). Again, a general hyperplane section is a cubic surface presenting 4 $A_1$ singularities.

\end{rmk}

Summing up what we have proved so far is the following:
\begin{prop}\label{sing-hs} Let $X$ be a cubic threefold singular along a curve $Y$ of degree $d$ and of $r$-th type, $r=1,2$. Then for a general hyperplane $H\subset \D{P}^4$, $S:=H\cap X$ 
	is a cubic surface with $d$ $A_r$ singularities.
\end{prop}

\section{Existence of smooth normal elliptic quintics}
In this section we use the notions previously introduced to show the existence of smooth elliptic quintics on certain normal cubic threefolds.
The method that we adopt to prove the existence of these curves relies on a deformation argument: it consists indeed in first proving that there exists a \textit{smooth} quintic elliptic curve $\mathcal{C}_0$ contained in a hyperplane section of $X$\DM{; then it is shown} that $\mathcal{C}_0$ deforms to a \DM{nondegenerate}  curve $\mathcal{C}$ \DM{in $X$}.
A smooth curve $\C$ obtained in this way is \DM{always} AG. Indeed, \DM{any smooth nondegenerate elliptic quintic $\C$ is ACM by Remark \ref{projnormcurve} and by the projective normality of the smooth curves of genus $g\geq 1$ embedded by complete linear systems of degree $\geq 2g+1$ (see \cite[Corollary from Theorem 6]{Mumford});} as moreover $\mathcal{O}_{\C}\simeq \omega_{\C}$, $\C$ is subcanonical, \DM{hence it is} AG.
\DM{
	We prove the following:}\smallskip
\begin{thm}\label{smooth_quintic}Let $X$ be \DM{a} normal cubic threefold that is not a cone and that satisfies one \DM{of the following two} conditions: 
	\begin{itemize}
		\item $X$ has isolated singularities;
		\item $\DM{\Sing}(X)$ contains a curve $Y$ that is either a line or a conic of first type or the union of three non-coplanar lines meeting at a point. 
	\end{itemize} 
	Then there exists a smooth \DM{nondegenerate} quintic elliptic curve $\C\subset X.$
\end{thm}

From \DM{Theorem} \ref{smooth_quintic} we will deduce:
\begin{prop}\label{Ulrich_smooth} Let $X$ be a cubic threefold satisfying one of the hypotheses of \ref{smooth_quintic}. Then there \DM{exists} a rank 2 skew-symmetric Ulrich bundle $\E\in \DM{\Coh\,}(X)$. 
\end{prop}
\DM{Consequently:}
\begin{cor}\label{cor-pf} Let $X$ be a cubic threefold satisfying \DM{one of} the hypotheses of \ref{smooth_quintic}. Then $X$ is Pfaffian.
\end{cor}

\paragraph{Illustration of the method in the smooth case}

\

The smooth case was treated in detail in \cite{Dima-Tikho}. \DM{As the proof of the singular case is obtained, in part, by an extension of the method of proof for the smooth case, we start by recalling it. So, first suppose} 
that $X$ is smooth. Applying Bertini's theorem, for a general hyperplane $H\subset \D{P}^4$, the intersection $S:=X\cap H$ is a smooth cubic surface.
A smooth cubic surface $S$ is a Del Pezzo surface of degree 3, hence it can be realized as the blowup of $\D{P}^2$ \DM{in} 6 points $p_1, \dots, p_6$ in \textit{general position}.
Denote by $\pi: S\to \D{P}^2$ the corresponding birational morphism and write $\Pic(S)= \D{Z}l \oplus_{i=1}^6 \D{Z} e_i$ where 
$l$ is the class of $\pi^*(L)$ for a generic line $L\subset \D{P}^2$ and the $e_i$s are the classes of the exceptional 
curves $E_0,\dots, E_6, \ E_i=\pi^{-1}(p_i)$.
\

Choose now 4 points $p_1,\dots p_4$ among the $p_i$s and consider the  linear system 

$$\D{P}^5\simeq|3l-p_1-p_2-p_3-p_4|\subset |\mathcal{O}_{\D{P}^2}(3)|
$$ of plane cubics passing through them.
For every $C_3 \in |3l-p_1-p_2-p_3-p_4|$ the strict transform $\pi^{-1}(C_3)$ is an elliptic \DM{curve (}this can be computed on $S$ by adjunction or from $g(\pi^{-1}(C_3))=g(C_3)=1$), belonging to the class $(3,-1,-1,-1,-1,0,0) \in \DM{\Pic\,}(S)$. 
The linear system $\DM|(3,-1,-1,-1,-1,0,0)\DM|$ is base-point free \DM(this can be deduced from the fact that, by \cite{Hartshorne} Ch. V Thm 4.6, the class $(3,-1,-1,-1,-1)$ is very ample on $\DM{\Bl}_{p_1,\dots p_4}\D{P}^2$); thus, by Bertini's theorem, a general element $\mathcal{C}_0\in \DM|(3,-1,-1,-1,-1,0,0)\DM|$ is smooth.
We compute that:
$$\DM{\deg\,}(\mathcal{C}_0)=(-K_s) \cdot\mathcal{C}_0=5$$
hence $\mathcal{C}_0$ is a smooth elliptic quintic contained in $S$.

\vspace{2mm}
In order to show that $\mathcal{C}_0$ deforms in $X$ to a \DM{nondegenerate} curve $\mathcal{C}$ we first need to check that $\DM{\Hilb}_X^{5n}$, the Hilbert scheme of elliptic quintics on $X$, is smooth at the point $[\mathcal{C}_0]$.
To this aim we consider the short exact sequence of sheaves on $\mathcal{C}_0$:

\begin{equation}\label{ses-normal}
0 \longrightarrow \mathcal{N}_{\mathcal{C}_0/S} \longrightarrow \mathcal{N}_{\mathcal{C}_0/X}\longrightarrow {\mathcal{N}_{S/X}}_{|{\mathcal{C}_0}}\longrightarrow 0.
\end{equation}
$\mathcal{N}_{\mathcal{C}_0/S}\simeq \mathcal{O}_{\mathcal{C}_0}(\mathcal{C}_0)$ is a line bundle of degree $\mathcal{C}_0^2=5$, the same holds for \mbox{$ {\mathcal{N}_{S/X}}_{|{\mathcal{C}_0}}\simeq \mathcal{O}_{\mathcal{C}_0}(1)$}.
Hence $h^0(\mathcal{N}_{\mathcal{C}_0/S})=h^0({\mathcal{N}_{S/X}}_{|{\mathcal{C}_0}})=5$, and $h^1(\mathcal{N}_{\mathcal{C}_0/S})=h^1({\mathcal{N}_{S/X}}_{|{\mathcal{C}_0}})=0$.
From these equalities we get that $h^1(\mathcal{N}_{\mathcal{C}_0/X})=0$, $h^0(\mathcal{N}_{\mathcal{C}_0/X})=10$, thus $\DM{\Hilb}_X^{5n}$ is smooth of dimension 10 at $[\C_0]$ (\cite{Gro}). The linear system $|\mathcal{C}_0|$ on $S=H\cap X$ has dimension 5; letting the hyperplane $H$ \DM{vary} in ${\D{P}^4}^{*}$, we see that the family of elliptic quintics contained in a hyperplane section of $X$ has dimension 9. As $\DM{\Hilb}_X^{5n}$ at $\mathcal{C}_0$ is smooth and of dimension 10, we can finally conclude that $\mathcal{C}_0$ deforms to a \DM{nondegenerate} smooth elliptic quintic $\mathcal{C}$.

\DM{\subsubsection* {The singular case} }
We now describe how the method that we have just presented \DM{also adapts} to certain singular cases (here we are still supposing that $X$ has no triple points).
The idea is the following: we look for a smooth elliptic quintic $\mathcal{C}_0$ contained in a hyperplane section $S:=X\cap H$ of $X$ and such that 
$\mathcal{C}_0 \cap \DM{\Sing}(S)=\emptyset$. 
The fact that $\mathcal{C}_0$ doesn't pass through any of the singular points of $S$, implies that $\mathcal{C}_0$ is entirely contained in $X_{sm}$, the smooth locus of $X$. We thus still get the short exact sequence of line bundles on $\mathcal{C}_0$ \ref{ses-normal} that allows us to compute again that $\DM{\Hilb}_X^{5n}$ at $\mathcal{C}_0$ is smooth and of dimension 10. 
Arguing as above, we conclude that $\mathcal{C}_0$ deforms, in $X_{sm}$, to a \DM{nondegenerate} smooth elliptic quintic $\mathcal{C}$. 
The existence of $\mathcal{C}_0$ is proved considering a minimal resolution of $S$:
$$ \phi: \tilde{S} \to S$$
and constructing a linear system on $\tilde{S}$ whose generic element $\tilde{\mathcal{C}}_0$ is mapped by $\phi$ onto a smooth elliptic quintic $\mathcal{C}_0:=\phi(\tilde{\mathcal{C}}_0)\subset S$.
From what had been formerly explained, for a general hyperplane $H\subset \D{P}^4$, the intersection $S:=X\cap H$ is a cubic surface with isolated singularities that are at most rational double points (RDPs for short).
A cubic surface $S\subset \D{P}^3$ presenting rational double points is a del Pezzo surface of degree 3.
A minimal resolution $\tilde{S}\to S$ is a so called \textit{weak Del Pezzo surface} of degree 3.

\

\subsection{ Del Pezzo and weak Del Pezzo surfaces}
We recall here some basic properties about Del Pezzo and weak Del Pezzo surfaces, focusing in particular on the case of surfaces of degree 3.
The main references for the section are \cite{CAG} and \cite{Pinkham}.
A \textit{Del Pezzo surface} is a \DM{nondegenerate} irreducible surface $S$ of degree $d$ in $\D{P}^d$ that is not a cone and that is not isomorphic to a projection of a surface of degree $d$ in $\D{P}^{d+1}$. It is possible to prove (see e.g. \cite{CAG} Thm. 8.1.11) that if $S$ is a del Pezzo surface, all its singularities are rational double points and its anticanonical sheaf $\omega_S$ is an ample line bundle. 
A cubic surface with at most RDP is an example of a Del Pezzo surface and since it is a degree 3 hypersurface in the 3-dimensional projective space, its canonical sheaf is $\omega_S^{-1}\simeq \mathcal{O}_S(1)$.

It is well known (this is proven e.g. in \cite{CAG} Thm. 8.1.15), that if $S$ is a \textit{singular} Del Pezzo surface, a minimal resolution of $S$, $\phi:\tilde{S}\to S$, is a so called \textit{weak Del Pezzo surface}. 
\begin{defn}A weak Del Pezzo surface is a smooth surface $\tilde{S}$ such that the anticanonical class $-K_{\tilde{S}}$ is nef and big.
	The degree $d$ of $\tilde{S}$ is defined as the number $d:=K_{\tilde{S}}^2$.
\end{defn}

Any weak Del Pezzo surface $\tilde{S}$ is isomorphic either to $\D{P}^1\times \D{P}^1$, to the Hizerbruch surface $\D{F}_2$ (both of degree 8) or to the surface $X_N$ (of degree $d=9-N$) obtained from $\D{P}^2$ by blowing up $N$ points, $0\le N \le 8$ in \textit{almost general position}, i.e. we have a map $\pi:\tilde{S}\to \D{P}^2$ that is a composition 
$$\tilde{S}\xrightarrow{\sim}X_N\xrightarrow{\pi_N} X_{N-1}\xrightarrow{\pi_{N-1}} \dots \xrightarrow{\pi_1} \D{P}^2:=X_0$$
where each morphism $X_{i}\xrightarrow{\pi_i} X_{i-1}$ is the blowup at a point $p_i \in X_{i-1}$.
We say that the points $p_1, \dots p_N$ lie in almost general position if at each step, the point $p_i$ does not belong to any irreducible curve in $X_{i-1}$ having self intersection $-2$. 


\

A weak Del Pezzo surface is a smooth Del Pezzo surface if and only if it is isomorphic either to $\D{P}^1\times \D{P}^1$ or to the blow-up $Bl_{p_1,\dots p_N}\D{P}^2$ of $\D{P}^2$ in $N, \ 0\le N\le 8$ points in \textit{general position} (cfr. \cite{CAG} Cor. 8.1.17). $N$ points $p_1, \dots p_N$ are said to be in general position if no three among them are collinear, no six points among them lie on a smooth conic and no cubic passes through all the points with one of them being a singular point.


\paragraph{The structure of the Picard group}
\

Let $\tilde{S}$ be a weak Del Pezzo surface 
isomorphic to the blow-up $\pi: X_N\to \D{P}^2$ of $\D{P}^2$ 
in $N$ points $p_1,\ldots, p_N$ in almost general position. The Picard group $\Pic(\tilde{S})$ of $\tilde{S}$ is the free abelian group of rank $N+1$:
$$\DM{\Pic\,}(\tilde{S})=\pi^*(\DM{\Pic\,}(\D{P}^2))\oplus \D{Z}e_1\oplus \dots \oplus \D{Z}e_N=\D{Z}l\oplus \D{Z}e_1\oplus \dots \oplus \D{Z}e_N,$$
with $l$ the class of $\pi^*(L)$ for a generic line $L\subset \D{P}^2$ and $e_i$ the class of the total transform in $\tilde{S}$ of $\pi_i^{-1}(p_i)$, the exceptional divisor of $\pi_i$. We compute immediately that the canonical class is:
$$K_{\tilde{S}}=-3l +e_1-\dots +e_N.$$

The intersection product defines a symmetric bilinear form on $\DM{\Pic\,}(\tilde{S})$ that, in the basis $(l,e_1,\dots e_N)$, is represented by the matrix 
$\DM{\diag}(1,-1,\dots, -1)$. This endows $\DM{\Pic\,}(\tilde{S})$ with the structure of a unimodular lattice of signature $(1,N).$
The orthogonal ${K_{\tilde{S}}}^{\perp}$ to $K_{\tilde{S}}$ is a sublattice of type $E_N$ and the restriction of the intersection product to this sublattice is non-degenerate. The group of automorphisms of $E_N$ preserving the intersection product is called the \textit{Weyl group} and it is usually denoted by $W(E_N)$ (see \cite{CAG} Ch. 8, Sect.2 for further details).

\subsubsection*{Negative curves on Weak del Pezzo surfaces}
Let $\tilde{S}$ be a weak Del Pezzo surface isomorphic to the surface $X_N, \ 1\le N\le 8$.
An irreducible and reduced curve $C\subset \tilde{S}$ is called a \textit{negative curve} if $C^2=-n, \ n>0$.
Applying adjunction formula and by the nefness of $K_{\tilde{S}}$, we compute that for a negative curve $C\subset \tilde{S}$, we must have $C^2=-n, \ n\in\{1,2\}$. 
In the former case we will call $C$ a $(-1)$-curve, in the latter $C$ is referred to as a $(-2)$-curve.
Still using adjunction we can prove that any negative curve $C\subset \tilde{S}$ has arithmetic genus $p_a(C)=0$, so that $C\simeq \D{P}^1$, and that if ever $C^2=-1$ then $K_{\tilde{S}}\cdot C=-1$, if $C^2=-2$ instead, $K_{\tilde{S}}\cdot C=0$.

It is shown in \cite{Pinkham}, Thm III.1. that $\tilde{S}$ is a smooth Del Pezzo surface 
if and only if the only negative curves on $\tilde{S}$ are $(-1)$-curves.
If we ask for the $N$-tuple $p_1,\ldots, p_N$ to satisfy the weaker condition of being in almost general position, $\tilde{S}$ can also contain $(-2)$-curves.

From now on, as a negative curve does not move in its linear equivalence class, we will identify it with its class in $\DM{\Pic\,}(\tilde{S})$.
Let now $D$ be any effective divisor with $D^2=-2$; write $D$ as $\sum_i n_i R_i$ with the $R_i$s effective and irreducible divisors in $\tilde{S}$, and $n_i\ge 0, \forall \ i$. 
Since 
$$0=D\cdot K_{\tilde{S}}=\sum_i n_i (R_i\cdot K_{\tilde{S}})$$ 
and as $-K_{\tilde{S}}$ is nef, we get that $R_i\cdot K_{\tilde{S}}=0, \ \forall\ i$; therefore, by adjunction, \mbox{$\DM{\deg\,}(K_{R_i})= R_i^2\ge -2$.} 
The $R_i$s are effective divisors, hence $R_i\cdot R_j\ge 0$ whenever $i\ne j$; 
these inequalities together with the condition $D^2=-2$ imply that \mbox{$R_i^2=-2, \forall \ i$.} 
So, summing up, each effective divisor $D$ with $D^2=-2$ has a (unique) representative  expressed as a linear combination, with non negative coefficients, of $(-2)$-curves.
An example of effective divisors with self intersection $-2$ is given by \textit{Dynkin curves}.
\begin{defn} A \DM{\textit{Dynkin curve}} on $\tilde{S}$ is a reduced connected curve $R$ such that its irreducible components $R_i$ 
	are $(-2)$-curves and the intersection matrix $(R_i\cdot R_j)$ is a symmetric integer matrix with: 
	$$(R_i\cdot R_j)\ge 0 \hspace{2mm} \text{for}\: i\ne j \hspace{1cm} \text{and} \hspace{1cm} (R_i\cdot R_i)=-2.$$
\end{defn}
To a \DM{Dynkin} curve $R$ we can associate a connected graph $\Gamma_{R}$, called \text{Dynkin-Coxeter diagram}. To each irreducible component $R_i$ of $R$ corresponds a vertex $v_i$, weighted by $R_i\cdot R_i$; the vertices $v_i$ and $v_j$ are connected by $R_i\cdot R_j$ edges.  
It is known (cfr. \cite{CAG} Sect.8.2) that a connected graph obtained in this way is of ADE type; more precisely of the types $A_n, \ D_n, \ E_6, \ E_7, \ E_8$.
We will call ADE type of a Dynkin curve $R$, the ADE type of the corresponding Dynkin-Coxeter diagram.

We conclude the section stating a useful criterion allowing to determine when an element $D\in \Pic(\tilde{S})$ such that $D^2=D\cdot K_{\tilde{S}}=-1$, corresponds to the class of a $(-1)-$curve.

\begin{prop}\label{crit_irred}[\cite{CAG}. Lemma 8.2.22]
	Let $\tilde{S}$ be a weak del Pezzo surface of degree $d> 1$ and let $D$ be a divisor class with $D^2=D\cdot K_{\tilde{S}}=-1.$ Then $D=E+R$ where $R$ is a non-negative sum of $(-2)$-curves, and $E$ is a $(-1)$-curve. Morevoer $D$ is a $(-1)$-curve if and only if for
	each $(-2)$-curve $R_i\subset \tilde{S}$, we have $R_i\cdot D \ge 0$. 
	
\end{prop}

\paragraph{Weak Del Pezzo surfaces of degree 3}\label{res_wdp}

\

From now on we will always deal with weak del Pezzo surfaces of degree $d=3$. Due to the classification reported at the beginning of the section, a degree $3$ weak Del Pezzo surface $\tilde{S}$ is obtained from $\D{P}^2$ blowing up 6 points in almost general position. Keeping the notations previously adopted, we have that $\Pic(\tilde{S})= \D{Z}l \oplus_{i=1}^6 \D{Z} e_i\simeq \D{Z}^7$ and $K_{\tilde{S}}=(-3,1,1,1,1,1,1)$, 
Morevoer since we are restricting to the case $d=3$ the following holds:
\begin{thm}[\cite{CAG} Thm 8.3.2] 
	
	\begin{itemize}
		\item $|-K_{\tilde{S}}|$ has no base points.
		\item Let $\mathfrak{R}$ be the union of $(-2)$-curves on $\tilde{S}$. Then $|-K_{\tilde{S}}|$ defines a regular map 
		$$\phi:\tilde{S}\to \D{P}(H^0(-K_{\tilde{S}})^*)\simeq \D{P}^d$$ 
		which is an isomorphism outside $\mathfrak{R}$. The image $S$ of this map is a Del Pezzo surface of degree $d$.
		The image of each connected component of $\mathfrak{R}$ is a rational double point of $S$.
	\end{itemize}
\end{thm}


The birational morphism $\phi:\tilde{S}\to S$ is a resolution of $S$ that contracts 
each \DM{Dynkin} curve $R \subset \tilde{S}$ to a RDP $p_{R}$ of $S$; $p_{R}$ is a a singularity of the type of the \DM{Dynkin} diagram of $R$. 
Given now a curve $\tilde{C}\subset \tilde{S}$, let $C \subset S$ be its image through $\phi$. $C$ passes through the singular point $p_{R}$ if and only if $\tilde{C}$ intersects a component of the \DM{Dynkin} curve $R$ contracted to $p_{R}$.
All the possible configurations of Dynkin curves on $\tilde{S}$, or equivalently all the possible configurations of RDPs on $S$ are the following (see for example \cite{Bruce_Wall}):
$$ A_1,\: A_2,\: A_3,\: A_4,\: A_5,\: 2A_1,\: A_1+A_2,\:2A_2,\: A_3+A_1,\: A_4+A_1,\: A_5+A_1$$
$$3 A_1,\: 2 A_1+ A_2,\: 2A_2+A_1,\: 3 A_2,\: A_3+2 A_1,\: D_4,\: D_5,\: E_6.$$

\subsection{Smooth normal elliptic quintics}

We illustrate now how to produce smooth elliptic quintics on a Del pezzo surface $S$ of degree 3. These will be of the form $\phi(\tilde{C})$ for a smooth curve $\tilde{C}\subset \tilde{S}$ such that $\tilde {C}^2=\tilde{C}\cdot (-K_{\tilde{S}})=5$, where $\phi:\tilde{S}\to S$ is a minimal resolution of $S$. 

\begin{prop}\label{bp_free}Let $\tilde{S}$ be a weak Del Pezzo surface of degree 3, $R\subset \tilde{S}$ a $(-2)$-curve and $E\subset \tilde{S}$ a $(-1)$-curve such that $R\cdot E=1$.
	Then the linear system $|R-K_{\tilde{S}}+ 2E|$ is base point free.
	
\end{prop}
\begin{proof}
	The first thing that we show is that $|R-K_{\tilde{S}} + 2 E|$ is nef.
	Given any irreducible effective divisor $D$ on $\tilde{S}$, we can write: 
	$D\cdot (R-K_{\tilde{S}} + 2 E) =
	(D\cdot R)+ (D\cdot (-K_{\tilde{S}}))+ (2 D\cdot E).$
	
	As $-K_{\tilde{S}}$ is nef, we have $D\cdot (-K_{\tilde{S}})\ge 0$. \DMN{Since} $R$ is a \DMN{$(-2)$-}curve\DMN,
	$D \cdot R < 0$ if and only if $R=D$. Similarly $E\cdot D<0$ if and only if $E= D$.
	Therefore $D\cdot (R-K_{\tilde{S}} + 2 E) \ge 0$ whenever  $D\ne E$ and $D\ne R$.
	
	Now, if $D=E$ we have: 
	\begin{equation}\label{int-2}
	D\cdot (R-K_{\tilde{S}} + 2 E)= 1+1-2=0.
	\end{equation}
	
	If $D=R$ we get: 
	\begin{equation}\label{int-1}
	D\cdot (R-K_{\tilde{S}} + 2 E)=-2+2=0. 
	\end{equation}
	We can thus  conclude that  $|R-K_{\tilde{S}} + 2 E|$ is nef. Given any divisor $D\in  |R-K_{\tilde{S}} + 2 E|$, we compute that
	$D\cdot (-K_{\tilde{S}})=5, \ D^2=D\cdot (-K_{\tilde{S}})= 5$. $D$ is nef, $D^2>0$ thus 
	$D$ is big and nef, consequently $D-K_{\tilde{S}}$ is big and nef as well. Applying Kawamata-Viewheg vanishing theorem we have $h^i(\mathcal{O}_{\tilde{S}}(D)=0$ for $i=1,2.$
	By Riemann-Roch we get $$\chi (\mathcal{O}_{\tilde{S}}(D))= h^0(\mathcal{O}_{\tilde{S}}(D))= \frac{1}{2}(D^2 -K_{\tilde{S}}\cdot D)+1= 6.$$ 
	
	We now prove that $|R-K_{\tilde{S}} + 2 E|$ has no fixed component. 
	\DMN{Indeed, write} $|R-K_{\tilde{S}} + 2 E|$ as $F+ |M|$ where $F$ and $|M|$ are, respectively the fixed and the mobile part of $|R-K_{\tilde{S}} + 2 E|$.
	$M$ is a nef effective divisor satisfying $M^2\ge 0$ and $\dim |M|=\dim |R-K_{\tilde{S}} + 2 E|=5$, \DMN{whilst} the linear system $|F|$ has dimension 0.
	By the fact that $M$ nef, we get that $M -K_{\tilde{S}}$ is big and nef, therefore still by Kawamata-Viewheg: $$\chi(\mathcal{O}_{\tilde{S}}(M))=h^0(\mathcal{O}_{\tilde{S}}(M))=6.$$
	By Riemann Roch $5= \frac{1}{2} (M^2 - M\cdot K_{\tilde{S}})$ hence $M^2=10+ M \cdot K_{\tilde{S}}= 5 - F\cdot K_{\tilde{S}}\ge 5$. 
	Now, by the nefness of $|R-K_{\tilde{S}}+2E|$ we get the inequality: 
	$$5=(R-K_{\tilde{S}} + 2 E)^2=(R-K_{\tilde{S}} + 2 E)\cdot (M+F)\ge (R-K_{\tilde{S}} + 2 E)\cdot M;$$
	whilst from the nefness of $M$ we deduce:
	$$(R-K_{\tilde{S}} + 2 E)\cdot M= (M + F)\cdot M=M^2+ M\cdot F\ge M^2.$$ 
	Therefore $$M^2=5=(R-K_{\tilde{S}} + 2 E)^2= (R-K_{\tilde{S}} + 2 E)\cdot (M+F)=$$ 
	$$=(R-K_{\tilde{S}} + 2 E)\cdot F + (M+F)\cdot M= (R-K_{\tilde{S}} + 2 E)\cdot F + M\cdot F + M^2.$$
	Again, the nefness of $R-K_{\tilde{S}} + 2 E$ and $M$ implies $(R-K_{\tilde{S}} + 2 E)\cdot F\ge 0, \  M\cdot F \ge 0$ so that a fortiori $(R-K_{\tilde{S}} + 2 E)\cdot F= 0, \  M\cdot F = 0$.
	Since $$5=(M+F)^2=M^2 +2 M\cdot F + F^2 = 5 + F^2$$ we get that $F^2=0,$ and as $M^2=5= 5- F\cdot K_{\tilde{S}}$, we deduce $F\cdot K_{\tilde{S}}=0$.
	As the intersection form is non-degenerate on ${(\D{Z}K_{\tilde{S}})}^{\perp}$, we can conclude that $F=0$\DMN{, and thus} $|R-K_{\tilde{S}} + 2 E|$ has no fixed component.
	By the fact that $|R-K_{\tilde{S}} + 2 E|$ has no fixed component, a general element $D$ in it can be written as a sum
	$D=\sum_{i}^m D_i$ where all the $D_i$s are numerically equivalent irreducible curves (this follows from \cite{Friedman}, Ex 5.11). This implies that $\forall\: i\ne j, (D_i-D_j)\cdot D_i=(D_i-D_j)\cdot D_j=0$ hence $D_i^2=D_j^2=D_i\cdot D_j$. Call $a= D_i^2$.
	We obtain $5=D^2= ma + 2(\binom {m}{2} a)= a( m + 2 \binom {m}{2})$. This can only occur for $m=1$ and $a=5$.
	Consider now the short exact sequence:
	\begin{equation}\label{ses-div}
	0 \rightarrow \mathcal{O}_{\tilde{S}} \rightarrow \mathcal{O}_{\tilde{S}}(D)\rightarrow \mathcal{O}_{D}(D)\rightarrow 0.
	\end{equation}
	Since $H^1(\mathcal{O}_{\tilde{S}})=0$, $H^0(\mathcal{O}_{\tilde{S}}(D))\to H^0(\mathcal{O}_{D}(D))$ is surjective so that if ever
	$|R-K_{\tilde{S}}+2E|$ has a base point $p$, we would have that $p\in D$ is a base point of its restriction to $D$.
	But	since $\mathcal{O}_{D}(D)$ is a line bundle of degree 5 on $D$, with $p_a(D)=1\ge g(D)$, it can't have base points (see for example \cite{Pinkham} IV Lemma 1).
\end{proof} 

This result implies that given a cubic surface with at most RDPs $S$ and a minimal resolution $\phi:\tilde{S}\to S$ of $S$, a curve $\C_0=\phi(\tilde{\C_0})$, where $\tilde{\C_0}$ is a general element of a linear system on $\tilde{S}$ of the form $|R-K_{\tilde{S}}+2E|$ with $R,$ and $E$ as in proposition \ref{bp_free}, is a curve having degree 5, arithmetic genus 1 and smooth outside $\mathcal{C}_0\cap \Sing(S)$. For any singular point $p\in \Sing(S)$, calling $R_p\subset \tilde{S}$ the Dynkin curve contracted to $p$, we have $\mult_p(\mathcal{C}_0)=R_p\cdot \tilde{\C_0}$. 
We show now that under the further assumption that $S$ presents at most $rA_1$ singularities, $r\le 3$, there always exists a curve $\tilde{\C_0}\subset \tilde{S}$ as above such that $\tilde{\C_0}\cdot R_p=0$ for any Dynkin curve $R_p\subset \tilde{S}$; accordingly $\mathcal{C}_0=\phi(\tilde{\C_0})$ is a smooth elliptic quintic on $S$ such that $\mathcal{C}_0\cap \Sing(S)=\emptyset.$ 

\begin{prop}\label{quintic-disjoint} Let $S$ be a cubic surface with $r\: A_1$ singularities, $r\le 3$. Then there exists a smooth quintic elliptic curve $\C_0\subset S$ 
	such that $\C_0\cap \Sing(S)=\emptyset$.
\end{prop}  
\begin{proof} 	We have already illustrated how to obtain elliptic quintics on a smooth cubic surface, hence we consider the cases where $1\le r \le 3$. Let $\phi:\tilde{S}\to S$ be a  minimal resolution of $S$. Denote by $p_1,\cdots p_r$ the singular points of $S$ and by $R_{p_1}, \dots R_{p_r}$
	the corresponding $(-2)-$curves on $\tilde{S}$ (since we are assuming that $S$ presents just $A_1$ singularities, each Dynkin curve on $\tilde{S}$ is a $(-2)$-curve). We claim the following:
	\begin{claim}\label{claim-ort} There always exists a $(-1)$-curve $E$ such that $E\cdot R_{p_1}=1$ and $E\cdot R_{p_i}=0$ for $i\ne 1$.
	\end{claim}
	We first show how, assuming the claim, the proposition follows. If such a $(-1)$-curve $E$ exists, we have that, by proposition \ref{bp_free}, the general element 
	$\tilde{\C}_0$ in $|R_{p_1}-K_{\tilde{S}}+2 E|$ is a smooth curve with ${\tilde{\C}_0}^2=5= \tilde{\C}_0\cdot (-K_{\tilde{S}})$.
	We compute that $$\tilde{\C}_0 \cdot R_{p_1}= R_{p_1}^2 + 2 E\cdot R_{p_1}=0, \hspace{2mm} \tilde{\C}_0 \cdot R_{p_i}= R_{p_1}\cdot R_{p_i}+ 2 E\cdot R_{p_i}=0 \ \text{for}\ i\ne 1.$$
	Therefore the curve $\C_0:=\phi(\tilde{\C_0})$ is a smooth elliptic quintic curve on $S$ disjoint from $\Sing(S)$.
	We end the proof of the proposition demonstrating the claim. Showing that a $(-1)$-curve $E$ satisfying the hypotheses of the claim exists is equivalent to showing that on $S$ there always exists a line $L$ passing through $p_1$ and such that $p_i \notin L$ for $i\ne 1$ (this is due to proposition \ref{crit_irred}, that establishes a bijection between $(-1)$-curves on $\tilde{S}$ and lines on $S$) .
	
	We choose coordinates $X_0, \dots, X_3$ on $\D{P}^3$ in such a way that $p_1=[1:0:0:0]$.
	Let $f \in \D{C}[X_0,\dots, X_3]$ be the homogeneous polynomial of degree 3 defining $S$. For our choice of coordinates $f$ can be written as 
	$$f(X_0,\cdots, X_3)= X_0 q(X_1,X_2,X_3) + c(X_1,X_2,X_3)$$
	where $q$ and $c$ are homogeneous forms of degree, respectively, 2 and 3. Since $p_1$ is an $A_1$ singularity, it is a conic node, hence $q$ is a quadric of rank 3.
	The intersection of the hyperplane $\{X_0=0\}$ with the cones having equations $\{q(X_0,X_1,X_2)=0\}$ and $\{c(X_0,X_1,X_2)=0\}$ is a zero-dimensional scheme $Z$ of length 6.
	All the lines contained in $S$ and passing through $p_1$ are of the form $\overline{p_1q}$ for $q \in Z$.
	A line $\overline{p_1q}, q\in Z$ passes through a singular point $p_i$ different from $p_1$ if and only if the conic $C_2:=\{X_0=q(X_1,X_2,X_3)=0\}$ and the plane cubic $C_3:=\{X_0=c(X_1,X_2,X_3)\}$ meet at $q$ with multiplicity greater then one. 
	Moreover the point $p_i$ will be an $A_{k-1}$ singularity where $k$ is the multiplicity of the intersecion of $C_2$ and $C_3$ at $q$.
	Since we are assuming that $S$ has at most r $A_1$ singularities, $C_2$ and $C_3$ meets in $r-1$ points with multiplicity $2$ implying that $Z$ consists of $r-1$ double points and of $6-2(r-1)$ simple points. Since $r\le 3$ there always exists then a point $q\in Z$ where $C_2$ and $C_3$ meets transversely so that $\overline{qp_1}$ is a line in $S$ satisfying
	$\overline{p_1q}\cap \Sing(S)=\{p_1\}$.

\end{proof}	 

\paragraph{Remark} Claim \ref{claim-ort} can also be proved by writing down explicitly an element $E \in \Pic(\tilde{S})$ such that $E^2=E\cdot K_{\tilde{S}}=-1$, intersecting with multiplicity one just one among the $(-2)$-curves and orthogonal to all the others, 
for all possible configuration of $(-2)$-curves leading to $r$ $A_1$ singularities. Such an element $E\in \Pic(\tilde{S})$ will then correspond to the class of a $(-1)$ curve due to proposition \ref{crit_irred}.
More precisely it is sufficient to show the existence of $E$ just on one representative $\tilde{S}$ of each $W(\bE_6)$ orbit of weak Del Pezzo surfaces of $r$ $A_1$ type. Indeed let $\tilde{S}, \ \tilde{S}'$ be a couple of weak del Pezzo surfaces containing, respectively, $r$ pairwise orthogonal $(-2)$-curves $\{R_1,\dots R_r\}$ and 
$\{R_1', \cdots R_{r}'\}$. Suppose the existence of an element $s \in W(\bE_6)$ inducing a bijection between $\{R_{1}, \cdots, R_{r}\}$ and $\{R_1', \cdots R_{r}'\}$. 
If $E\in \Pic(\tilde{S})$ is such that $$E^2=E\cdot K_{\tilde{S}}=-1, \hspace{2mm}\text{and} \ E\cdot R_1=1, \ E\cdot R_i=0\ \text{for} \ i\ne 1,$$ then $s(E)$ satisfies  $${s(E)}^2=s(E)\cdot K_{\tilde{S'}}=-1, \hspace{2mm} \text{and} \ s(E)\cdot s(R_1)=1, \ s(E)\cdot s(R_i)=0 \ \text{for} \ i\ne 1,$$ hence $s(E)$ is the class of a $(-1)$-curve (still by \ref{crit_irred}) meeting just one $(-2)$-curve. 

This is particularly easy for $r=1$. Consider indeed $\tilde{S}$ and $\tilde{S'}$ each of them containing just a $(-2)$-curve $R\subset \tilde{S}, \ R'\subset \tilde{S'}$. The Weyl group $W(E_6)$ acts transitively on the set of \textit{roots} of $\tilde{S}$ (see \cite{CAG} Thm. 8.2.14), namely elements $D\in \Pic(\tilde{S})$ such that $D\cdot K_S=0, \ D^2=-2$, therefore there always exists $s\in W(\bE_6)$ such that $s(R)=R'$. 
Without loss of generality we can thus assume that $\tilde{S}$ is obtained as the blowup: 
$$\tilde{S}\simeq \Bl_{\{p_1,\dots,p_6\}}\D{P}^2, \hspace{3mm} p_i \in \D{P}^2, \hspace{3mm} |h-p_1-p_2-p_3|\ne \emptyset$$ 
of 6 proper points $p_1,\dots p_6$ on $\D{P}^2$ with just 3 of them, say $p_1, \ p_2, \ p_3$ lying on a line $L$. 
The strict transform of $L$ is a $(-2)-$curve $R$ (represented by the vector $(1,-1,-1,-1,0,0,0)$ in $\Pic(\tilde{S}))$), each exceptional divisor $E_i=\pi^{-1}(p_i), \ i=1,2,3$ is a $(-1)-$curve such that $E_i\cdot R=1$. Hence the generic elements of the classes 
$$(4, 0, -2, -2, 0, 0, 0), \hspace{3mm} (4, -2,0,-2,0,0,0), \hspace{3mm} (4, -2,-2,0,0,0,0)$$
are mapped by $\pres$ to a smooth elliptic quintic lying in the smooth locus of the surface $S:=\pres(\tilde{S})$.

We can now prove Theorem \ref{smooth_quintic}:

\begin{proof}[Proof of Theorem \ref{smooth_quintic}]
	Let $X$ be a cubic threefold satisfying the hypotheses of theorem \ref{smooth_quintic}. Then, by proposition \ref{quintic-disjoint}, on a general hyperplane section
	$S:=H\cap X$ there always exists a smooth elliptic quintic $\C_0$ disjoint from $\Sing(S)$. The condition  \mbox{$\C_0\cap \Sing(S)=\emptyset$} implies that $\C_0$ is entirely contained in $X_{sm}$, the smooth locus of $X$. The argument used in the smooth case applies verbatim to $X_{sm}$ so that we conclude that $\mathcal{C}_0$ deforms in $X_{sm}$ to a \DM{nondegenerate}  smooth elliptic quintic $\mathcal{C}\subset X$.
\end{proof}

Applying theorem \ref{smooth_quintic} we can prove the following:
\begin{prop}[Prop. \ref{Ulrich_smooth}] Let $X$ be a cubic threefold satisfying the hypotheses of \ref{smooth_quintic}. Then there exists a rank 2 skew-symmetric Ulrich bundle $\E\in \DM{\Coh\,}(X)$. 
\end{prop}
\begin{proof} Let $X$ be a cubic threefold satisfying the hypotheses of theorem \ref{smooth_quintic}. From what we have just proved, there exists then a smooth quintic elliptic curve $\C\subset X$ that is not contained in any hyperplane section of $X$. 
	Applying Serre's correspondence, the class $1 \in H^0(\mathcal{O}_{\C})\simeq \DM{\Ext}^1(\mathcal{I}_{\C}, \mathcal{O}_X(-2))$ gives rise to a short exact sequence
	$$0\longrightarrow \mathcal{O}_X(-2)\longrightarrow \mathcal{N}\longrightarrow \mathcal{I}_{\mathcal{C}}\longrightarrow 0$$
	where $\mathcal{N}$ is a rank 2 ACM sheaf. Localizing at a point $x\in \C$ we get:
	$$0\longrightarrow \mathcal{O}_{X,x}\longrightarrow \mathcal{N}_x\longrightarrow \mathcal{I}_{\mathcal{C},x}\longrightarrow 0,$$
	corresponding to $1_x  \in \DM{\Ext}_{\mathcal{O}_{X,x}}^1(\mathcal{I}_{\mathcal{C},x}, \mathcal{O}_{X,x})$.
	Now, since $\C$ is smooth, we have: 
	$$\D{C}\simeq H^0(\mathcal{O}_{\C})\simeq  H^0( \sheafext^1(\mathcal{I}_{\mathcal{C}}, \mathcal{O}_X(-2))\simeq \DM{\Ext}^1(\mathcal{I}_{\mathcal{C}}, \mathcal{O}_X(-2)).$$
	Hence the class $1 \in H^0(\mathcal{O}_{\C})$ does not vanishes on any point of $\C$, so that for all points \mbox{$x \in \C$}, $1_x$ generates $\mathcal{O}_{\C,x}\simeq \sheafext_x^1(\mathcal{I}_{\mathcal{C}}, \mathcal{O}_X(-2))\simeq \DM{\Ext}_{\mathcal{O}_{X,x}}^1(\mathcal{I}_{\mathcal{C},x}, \mathcal{O}_{X,x}).$
	These facts, together with propositions \ref{pd_1} and \ref{Serre}, implies that $\mathcal{N}_x$ is free. Therefore $\mathcal{N}$ and $\E:=\mathcal{N}(2)$ are both vector bundles; by proposition \ref{nondeg_Ulrich} and remark \ref{Ulrich_lf} we can finally conclude that $\E$ is a skew-symmetric Ulrich bundle. 
\end{proof}

\section{Existence of AG normal elliptic quintics}

We describe now a constructive method to prove the existence of an AG normal elliptic quintic curve $\C \subset X$.
In order to apply this method we need the following pair of geometric objects:\smallskip
\begin{itemize}
	\item A rational quartic curve (not necessarily irreducible) $C_4 \subset X$.
	\item A cubic scroll $\pi:\Sigma\to \D{P}^1$ containing $C_4$ and such that the class of $C_4$ in $\DM{\Pic\,}(\Sigma)$ is $C_4 \sim D +3F$, \DM{where} $D$ and $F$ \DM{are}, respectively, the classes of the directrix and of a fiber $\pi^{-1}(x), \ x\in \D{P}^1$.
	
\end{itemize}
Recall that a cubic scroll $\Sigma\to \D{P}^1$ is a rational ruled surface whose Picard group is a free abelian group of rank 2, $\DM{\Pic\,}(S)= \D{Z}D\oplus \D{Z}F$. The directrix $D$ is a rational $(-1)$-curve 
meeting transversely each fiber, so that the intersection number $D\cdot F$ is equal to one.
As two distinct fibers don't meet, $F$ has self intersection equal to zero. From these facts we can compute, by adjunction, that the canonical class is $K_{\Sigma}\sim -2D-3F$ and that the class of a hyperplane section is $H\sim D+2F$.
Now, if there exists a rational quartic curve $C_4$ contained in $\Sigma$, for a pair of non negative integers $a, \  b$ we can write $C_4 \sim aD+bF$. 
Since $C_4$ has genus 0 and degree 4 we must have:
$$C_4 \cdot H=4 \hspace{2mm} \text{and} \hspace{2mm} (C_4 +K_{\Sigma})\cdot C_4=-2.$$ Hence $a$ and $b$ must satisfy: 
$$ (aD+bF)\cdot (D+2F)=a+b=4, \hspace{3mm} ((a-2)D+(b-3)F)\cdot (aD+bF)=2b(a-1)-a(a+1)=-2.$$
The only two pairs of non negative integers satisfying these equations are $(2,2)$ and $(1,3)$.  
Suppose now that there exists $C_4\subset (X\cap\Sigma)$ a rational quartic such that in $\DM{\Pic\,}(\Sigma)$ we have $C_4 \sim D+3F$. 
The intersection $X\cap\Sigma$ \DM{gives} a divisor on $\Sigma$ whose class in $\DM{\Pic\,}(\Sigma)$ is $3 H= 3D+6F$; hence the curve $\C$, residual to $C_4$ in $X\cap H$ is an elliptic quintic such that $\C\sim 2D+3F\sim -K_{\Sigma}$. 
Such a curve $\C$ should necessarily be \DM{nondegenerate}, as
$$H^0( \Sigma, \mathcal{I}_{\C/\Sigma}(1))= H^0(\mathcal{O}_{\Sigma}(H- \C)) =H^0(\mathcal{O}_{\Sigma}(-D-F))=0;$$
since $\C \sim -K_{\Sigma}$, we can assert that moreover $\C$ is an AG curve due to the following:

\begin{thm}\label{AG_anticanonical}[\cite{Miro_Roig}Thm 1.3.11] Let $\Sigma$ be an ACM surface in $\D{P}^4$ satisfying the condition $G_1$ (Gorenstein in codimension 1). Let $K_{\Sigma}$ be the canonical divisor and let $D$ be an effective divisor linearly equivalent to $mH-K_{\Sigma}$ for some $m\in \D{Z}$. Then $D$ is an AG curve whose canonical sheaf $\omega_D$ is isomorphic to $\mathcal{O}_D(m)$. 
\end{thm} 

The condition $G_1$ consists in requiring that each local ring $\mathcal{O}_{\Sigma,x}$ of dimension \DM{less than} or equal to 1 is Gorenstein.

We see that a cubic scroll $\Sigma$ satisfies the hypotheses of the theorem. Indeed it is an ACM surface ( see \cite{Miro_Roig}, proof of thm 1.3.15 for a classification of smooth rational ACM surfaces); 
as moreover $\Sigma$ is smooth, every local ring is regular hence Gorenstein. Accordingly the curve $\C$ obtained as the residual to $C_4$ in $X\cap \Sigma$ is an
AG elliptic quintic curve 
such that $\mathcal{O}_{\C}\simeq\omega_{\C}$. 

In the next section we will prove the existence of  $C_4$ and $\Sigma$ as above on a normal cubic threefold $X$ singular along a line or a conic of the second type or along a (possibly reducible) rational quartic (however the same approach can be applied to other normal threefolds with no triple points). 
This will lead us to conclude the prove of:
\begin{thm*}[Theorem \ref{normal_quintic}] Let $X$ be a normal cubic threefold that is not a cone. Then $X$ contains a non-degenerate AG quintic elliptic curve $\C.$
\end{thm*}

We will then show that also this time, the Ulrich sheaves constructed from these curves by mean of Serre's correspondence are locally free (so that they are ``naturally'' skew-symmetric).

\subsubsection*{Illustration of the method in the smooth case}
This time again we first describe how the method works in the smooth case and then we will show how we can adapt it to singular threefolds, where we will need to consider also reducible quartic curves.
Suppose that $X$ is a smooth cubic threefold.
It was proved in \cite{Dima-Tikho} that on $X$ there always exists a smooth rational quartic curve $C_4 \subset X$.

Starting from \DM{a smooth $C_4$,} we can construct a cubic scroll $\Sigma$ as follows:
we take $D$, a chord of $C_4$, $D=\overline{x_1x_2}, \ x_i\in C_4$. 
We fix two points $x_3, \ x_3'$ with $x_3 \in C_4$, $x_3' \in D$.
We consider then all the lines of the form $\overline{zz'}$ where $z \in C_4, \ z' \in D$ \DM{is} a pair of points satisfying the cross-ratio equation:
$$[x_1,x_2,x_3,z]=[x_1,x_2,x_3',z'].$$

The lines $\overline{zz'}$ obtained in this way draw a cubic scroll $\Sigma\subset \D{P}^4$.
The directrix of $\Sigma$ is exactly the line $D$; the curve $C_4$ intersect $D$ in two points, namely $x_1, \ x_2$ hence $\Gamma \sim D+3F$ in $\DM{\Pic\,}(\Sigma)$.

\subsubsection*{The singular case: construction of the rational quartic $C_4$}
We will now adapt the method that we have just described to the case where $X$ is a normal threefold that is singular either along a line or a conic of the second type, or along a rational quartic. This is done starting from a reducible and non-degenerate rational quartic $C_4$ on $X$ of the form $C_4=C_2\cup l\cup l'$ for a smooth conic $C_2\subset X$ and a pair of skew lines $l, \ l'$ meeting $C_2$ and not lying in $\langle C_2 \rangle$. 

The first step to obtain $C_4$ is the construction of $C_2\cup l$, a rational cubic lying on the surface
$X\cap \langle C_2 \cup l\rangle$.
In the next paragraph we will show in detail how to obtain rational cubics of such a kind on a general hyperplane section of $X$, a cubic surface that (by prop. \ref{sing-hs}) presents singularities of types $A_2, \ 2 A_2$ or $4 A_1$.
\
In more general terms given a cubic surface $S$, whenever there exists a pair of dijoint lines $l_a, \ l_b$ on $S$, it is possible to obtain a 
rational cubic $C_3=C_2\cup l_b$ for a smooth conic $C_2$ such that $l_a\subset \langle C_2 \rangle$ and $\mult_p(C_2)=\mult_p(l_a), \ \forall \ p\in \Sing (S)$.

Consider indeed a line $l_a\subset S$ and denote by $\mathcal{H}_{l_a}:=\D{P}(H^0(\mathcal{I}_{l_a/\D{P}^3}(1)))\simeq \D{P}^1$, the pencil of planes containing $l_a$.
For every $H\in \mathcal{H}_{l_a}$, the intersection $H\cap S$ is a plane cubic consisting of the union of $l_a$ and a conic $C_{2,H}$.
The Gauss map:
\begin{align*}
\gamma: S \dashrightarrow & S^{*}\subset {\D{P}^{3^*}} \\
x \mapsto & \D{T}_x S
\end{align*} 
mapping a (general) point $x \in S$ to the projective tangent space $\D{T}_xS$ to $S$ at $x$, induces a degree 2 rational map: 
$$\gamma\restriction_{l_a}:l_a\dashrightarrow \mathcal{H}_{l_a}$$
determined by a linear system $\mathcal{Q}_{l_a}\subset |\mathcal{O}_{l_a}(2)|$ of quadrics on $l_a$. For a general \mbox{$H\in \im(\gamma\restriction_{l_a})$}, $\gamma\restriction_{l_a}^{-1}(H)$ is the pair of points defined by $C_{2,H}\cap l_a$. 
Whenever \mbox{$l_a\cap \Sing(S)=\emptyset$} (in what follows we will refer to a line disjoint from the singular locus as a \textit{regular line}), $\mathcal{Q}_{l_a}$ is a base point free pencil 
that can be generated by 2 distinct quadrics of rank one (defined by the intersection of $\mathcal{Q}_{l_a}\simeq \D{P}^1$ with the conic parameterizing rank 1 quadrics on $l_a$). 
We thus have 2 planes $H', \ H''$ such that $l_1$ is tangent to $C_{2,H'}$, $C_{2,H''}$ respectively, whilst for $H$ general, $\gamma\restriction_{l_a}^{-1}(H)$ consists of 2 distinct points.

If $l_a$ passes trough just one singular point $p_0$, all quadrics in $\mathcal{Q}_{l_a}$ passes through $p_0$ (namely $\mathcal{Q}_{l_a}$ is tangent to the conic parameterizing rank 1 quadrics on $L_1$), 
hence $\forall H \in \mathcal{H}_{l_a}$, $C_{2,H}$ passes through $p_0$.

Finally if $l_a$ is generated by two singular points $p_0, \ p_1$, there exists a plane $H$ tangent to $S$ all along $l_a$, that is $\dim(\mathcal{Q}_{l_a})=0$ and $\im(\gamma\restriction_{l_a})$ consists just of $H$. 


Consider now a plane $H\in \mathcal{H}_{l_a}$; the singular points of the plane cubic $H\cap S=C_{2,H}\cup l_a$ are located by:
$$\Sing(S\cap H)= (C_{2,H}\cap l_a)\cup (\Sing(C_{2,H})\setminus l_a).$$

This means that a plane $H\in \mathcal{H}_{l_a}$ is a so called \textit{bitangent plane to $S$}, that is there exists only two points $x_1, \ x_2 \in S$ such that $H\subset \D{T}_{x_i}S$, if and only if $C_{2,H}$ is smooth.

The existence of at least 3 points $x_i \in S, \ i=1,2,3$ such that $H\subset \D{T}_{x_i}S$ (so that $H$ is at least a \textit{tritangent plane} to $S$), is instead equivalent to the fact that $C_{2,H}$ degenerates to the union of two lines. (More precisely $H$ is tritangent if and only if $C_{2,H}$ is the union of two distinct lines both different from $l_a$ whilst $H$ is tangent to $S$ along the entire $l_a$ if and only if $l_a \subset C_{2,H}$).
For a plane $H\in \mathcal{H}_{l_a}$ that is at least tritangent, $H\cap S$ is a so called \textit{triangle}, namely a plane cubic given by the union of three lines.
Since a line on a cubic surface is contained in a finite number of triangles (at most 5), we get that for a general $H\in \mathcal{H}_{l_a}$, $C_{2,H}$ is smooth. 
Summing up, for such a $H$ we have:
$$ \D{T}_{x}(S\cap H)=H \iff x\in l_a\cap C_{2,H},$$
so that:
$$
\Sing(S)\cap H \subset l_a  \cap C_{2,H},\hspace{5mm} \mult_p(C_{2,H})=\mult_p(l_a), \ \forall p\in \Sing(S).
$$

(Note that conversely whenever we start from a smooth conic $C_2\subset S$, the intersection $\langle C_2 \rangle \cap S$ is a plane cubic of the form
$C_2 \cup l$, for a  line $l \subset S$; the smoothness of $C_2$ ensures that $\langle C_2 \rangle$ is a bitangent plane so that $H\cap \Sing(S)\subset C_2 \cap l$).

Now, if we consider a line $l_b\subset S$ disjoint from $l_a$, every $H\in \mathcal{H}_{l_a}$ will meet $l_b$ in exactly one point lying on $C_{2,H}$.
From what we have just illustrated we can conclude that for a general point $x \in l_b$, the plane $H:=\langle x,l_a\rangle $ is such that $C_2:=C_{2,H}$ is a smooth conic and $C_3:=C_{2}\cup l_b$ is a 
rational cubic such that $\mult_p(C_{2,H}\cup\: l_b)=\mult_p(l_a\cup\: l_b), \ \forall p\in \Sing(S)$.
Once we have obtained $C_3$, the construction of $C_4$ is almost immediate since it is easy to show that for a general point in $C_2$ there exists a line passing through it and not constained in $\langle C_3\rangle$.


\paragraph{\textbf{$\bullet$ $S=X\cap H$ has $4 A_1$ singularities}}

\

Let's suppose that $X$ is singular along a rational quartic (that can possibly degenerate to a  couple of conics) so that for a general hyperplane $H\subset \D{P}^4$, $S:=X\cap H$ is a cubic surface with 4 singular points $p_0, \ldots, p_3$ of type $A_1$. 
The surface $S$ contains a total of 9 lines:
\begin{itemize}
	\item we have the six lines $\overline{p_ip_j}, \ i,j \in \{0,\ldots 3\}, \ i<j,$ joining two singular points;
	\item there exist then 3 regular lines $l_{0123}, \ l_{0213}, \ l_{0312}$ constructed as follows. For every line $\overline{p_ip_j}$, there exists a plane $\Delta_{ij}$ tangent to $S$ all along $\overline{p_ip_j}$ and intersecting $S$ along a plane cubic union of $\overline{p_ip_j}$ (counted with multiplicity 2) with a regular line,
	this latter located by all the intersections of the form $\Delta_{ij}\cap \Delta_{kl}$, for every line $\overline{p_kp_l}$ disjoint from $\overline{p_ip_j}$.
	We have three pairs of disjoint lines spanned by singular points, the pairs of the form $\overline{p_0p_i}, \ \overline{p_jp_k}, \ i\in \{1,2,3\}, \ j<k, \ \{0,i\}\cap \{j,k\}=\emptyset$. Each intersection $\Delta_{0i}\cap \Delta_{jk}$ defines one among the three aforementioned lines $l_{0ijk}$.
	It easy to notice that furthermore these three lines lie on a same plane $\Delta_r$. 
\end{itemize}

According to the characterization we have just presented, we can see that a line $\overline{p_ip_j}$ intersects $\overline{p_kp_l}$ if and only if $\{i,j\}\cap \{k,l\}\ne\emptyset$ and that it meets only one among the regular lines, the residual to $2\: \overline{p_ip_j}$ in $\Delta_{ij}\cap S$.
$\overline{p_ip_j}$ belongs to three triangles:
\begin{itemize}
	\item $\overline{p_ip_j}\cup \overline{p_ip_k}\cup \overline{p_jp_k}, \ k\ne i,j$;
	\item $\Delta_{ij}\cap S.$
\end{itemize}

Similarly we can observe that a regular line $l_{0ijk}$ intersect 4 lines: the other two regular lines (since regular lines are coplanar) and the lines $\overline{p_0p_i}, \ \overline{p_jp_k}$. Each regular line $l_{0ijk}$ belongs to three triangles as well:
\begin{itemize}
	\item $\Delta_r\cap S= l_{0123}\cup l_{0213}\cup l_{0312}$;
	\item $\Delta_{0i}\cap S=2 \overline{p_0p_i}\cup l_{0ijk}$ and $\Delta_{jk}\cap S= 2\overline{p_jp_k}\cup l_{0ijk}$.
\end{itemize}
\

In order to apply our construction of $C_3$ we need a pair of disjoint lines: these are either of the form $\overline{p_0p_i},\ \overline{p_jp_k}$ with $ \{0,i\}\cap\{j,k\}=\emptyset$, or of the form $l_{0ijk}, \:\overline{p_lp_m}$, for $\#(\{l,m\}\cap \{j,k\})=1$. 
Starting from a pair of disjoint lines spanned by singular points, take 
for simplictiy the lines $\overline{p_0p_1}$ and $\overline{p_2p_3}$, we have that for $x\in \overline{p_2p_3}, \ x\ne p_2,p_3$, \mbox{$ x\notin l_{0123}$}, the plane $H:=\langle x,\overline{p_0p_1}\rangle$ is bitangent to $S$ so that $C_{2,H}$ is a smooth conic meeting $\overline{p_0p_1}$ in $p_0, \ p_1$. The cubic $C_3=C_{2,H}\cup \overline{p_2p_3}$ satisfies $\mult_{p_i}(C_3)=1, \ \forall\: i=0,\ldots 3$.

Take now a regular line and another line disjoint from it. Up to permuting the singular points we may assume that these are $l_{0123}$ and $\overline{p_0p_2}$.
For $x\in \overline{p_0p_2}, \ x\ne p_0,p_2$, \mbox{$ x\notin (\Delta_r\cap \overline{p_0p_2})$}, the plane $H:=\langle x,l_{0123}\rangle$ is bitangent to $S$ and $C_{2,H}$ is a smooth conic contained in the smooth locus of $S$. $C_3:=C_{2,H}\cup \overline{p_0p_2}$ is a non-degenerate rational cubic passing, with multiplicity one, just trough 2 of the singular points, that are $p_0$ and $p_2$.
Considering a general point $x \in l_{0123}$, namely $x$ must differs from $l_{0123}\cap\overline{p_0p_1}, \ l_{0123}\cap \overline{p_2p_3}$ and  $l_{0123}\cap l_{0213}$,
we get a bitangent plane $H:=\langle x, \overline{p_0p_2}\rangle$, 
a smooth conic $C_{2,H}$ passing throug $p_0,\ p_2$ and consequently the rational cubic $C_{2,H}\cup l_{0123}$ that, again, meets $\Sing(S)$ in $p_0,\ p_2$.


\paragraph{$\bullet$	$S:=X\cap H$ has an $A_2$ singularity}

\

Let $X$ be a cubic threefold singular along a line of the second type, a general hyperplane setion $S=X\cap H$ is a cubic surface with a singular point $p_0$ of type $A_2$. 
There exist 15 lines on $S$, again we recall  briefly their geometry.
We have 6 lines $l_1, \ldots l_6$ passing through the singular point $p_0$ and consisting of two 3-uples $l_1, \ l_2, l_3$ and $l_4,\ l_5, l_6$ of coplanar lines. We call $\Delta_{123}, \ \Delta_{456}$ the planes spanned by these 3-uples.
Every plane $\langle l_i, l_j\rangle, \ i\in \{1,2,3\}, \ j\in\{4,5,6\}$ intersect $S$ along a plane cubic union of $l_i, \ l_j$ and of a third line $l_{ij}$ entirely contained in the smooth locus of $S$. Applying this construction to every couple $ l_i,\  l_j, \ i\in \{1,2,3\}, \ j\in\{4,5,6\}$ we get 9 regular lines $l_{ij}$.
It is easy to see that a line $l_i$ meets a regular line $l_{jk}$ if and only if $i \in \{j,k\}$ and that two regular lines $l_{ij}, \ l_{kl}$ intersects if and only if $\{ij\}\cap\{kl\}=\emptyset$. 
A line $l_i\subset \Delta_{123}$ (resp. $l_i\subset \Delta_{456}$) belongs to 4 triangles: 
\begin{itemize} 
	\item $\Delta_{123}\cap S=l_1\cup l_2\cup l_3$ (resp. $\Delta_{456}\cap S=l_4\cup l_5 \cup l_6$);
	\item $l_i\cup l_j \cup l_{ij}$ for $l_j \subset \Delta_{456}$ (resp. $l_i\cup l_j \cup l_{ij}$ for $l_j \subset \Delta_{123}$).
\end{itemize}
A regular line $l_{ij}$ belongs to 3 triangles:
\begin{itemize}
	\item $l_{ij}\cup l_i\cup l_j$;
	\item $l_{ij}\cup l_{kl} 	\cup l_{mn}, \ k,m\in \{1,2,3\}, \ k,m\ne i, \ l,n\in\{456\}, \ l,n\ne j$.
\end{itemize}
A pair of disjoint lines on $S$ is either of the form $l_i, \ l_{jk}, \ i\not \in \{j,k\}$ or of the form $l_{ij}, \ l_{jk}$. 
Let's start from a pair of the first kind, $l_1, \ l_{24}$, let's say. For $x\in l_{24}, \ x\ne l_{24}\cap \ l_2$, \mbox{$ x\ne l_{24}\cap l_4$},  $x\ne l_{24}\cap l_{1j}, \ j=5,6$, we get a bitangent plane $H:=\langle x, l_1\rangle$ and a smooth conic $C_{2,H}$ passing through $p_0$. 
Similarly we get a bitangent plane $H:=\langle x, l_{24}\rangle$ for $x\in l_1$, $x\ne p_0, \ x\ne l_1\cap l_{1j}, \ j\in\{5,6\}$, and consequently a smooth conic $C_{2,H}$ contained in the smooth locus of $S$ (notice that in both cases we obtain a cubic $C_3$ that passes through $p_0$).
Considering a pair of regular disjoint lines, take for example $l_{14}, \ l_{15}$, we get non degenerate cubic $C_3=C_{2,H}\cup l_{15}$ (resp. $C_{2,H}\cup l_{14}$), for a smooth conic $C_{2,H}$ obtained from the bitangent plane $H:=\langle x, l_{14}\rangle$, (resp. $H:=\langle x, l_{15}\rangle$), with $x\in l_{15}$, \mbox{$ x\ne l_{15}\cap l_{1}$},  $x\ne l_{15}\cap l_{i6}, \ i\in\{2,3\}$ (resp. with $x\in l_{14}, \ x\ne l_{14}\cap l_{1}, \ x\ne l_{14}\cap l_{i6}, \ i\in\{2,3\}$.)
In this case $C_3$ does not pass through $p_0$.

\paragraph{$\bullet$ $S=X\cap H$ has $2 A_2$ singularities}

\

We finally pass to the case where $X$ is singular along a conic of the second type so that a general hyperplane section $S=X\cap H$ is a cubic surface presenting two singular points $p_0, \ q_0$ that are $A_2$ singularities.
On such a cubic surface we have 7 lines. We have indeed 4 lines passing through $p_0$: the line $\overline{p_0q_0}$ and three other coplanar lines $\overline{p_0u_i}, \ i=1,2,3$ for three smooth alligned points $u_1, \ u_2, u_3$ on $S$. 
The remaining three lines are the residuals to $\overline{p_0q_0}\cup\overline{p_0u_i}$ in $\langle p_0, q_0, u_i\rangle\cap S$; these are 3 lines $\overline{q_0v_i}, \ i=1,2,3$ passing through $q_0$ and not through $p_0$.
From this description we see that $\overline{p_0q_0}$ meets every other line, each line $\overline{p_0u_i}$ (resp. $\overline{q_0v_i}$) meets $\overline{p_0u_j}, \ i\ne j$ at $p_0$ (resp. $ \overline{q_0v_j}, \ i\ne j$ at $q_0$) and the line $\overline{q_0v_i}$ (resp. the line $\overline{p_0u_i}$).
The line $\overline{p_0q_0}$ belongs to the three triangles $\overline{p_0q_0}\cup \overline {p_0u_i}\cup \overline{q_0v_i}, \ i\in\{1,2,3\}$.
A line $\overline{p_0u_i}$ through $p_0$ (resp. $\overline{q_0v_i}$ through $q_0$) fits into two triangles:
\begin{itemize}
	\item $\overline{p_0u_1}\cup\overline{p_0u_2}\cup\overline{p_0u_3}$ (resp. $\overline{q_0v_1}\cup\overline{q_0v_2}\cup\overline{q_0v_3}$);
	\item $\overline{p_0u_i}\cup \overline{p_0q_0}\cup \overline{q_0v_i}.$
\end{itemize}
The only pairs of disjoint lines are of the form $\overline{p_0u_i}, \ \overline{q_0v_j}, \ i\ne j$. Taking then a point $x \in \overline{q_0v_j}$ such that $x\ne q_0, \ x\not\in \overline{p_0u_j}$, we get a bitangent plane $H:=\langle x, \overline{p_0u_i}\rangle$ and consequently a smooth conic $C_{2,H}$ passing throug $p_0$. 
Applying an equivalent argument for a general $x \in \overline{p_0u_i}$ we get a bitangent plane $H:=\langle x,\overline{q_0v_j}\rangle$ and a smooth conic $C_{2,H}$ through $q_0$.
In both cases we obtain cubics passing trough both singular points.

\

We can now come back to the construction of rational quartics on the cubic threefold $X$.
We have just seen that for a general hyperplane $H\subset \D{P}^4$, on the cubic surface $S:=X\cap H$ there always exists a rational cubic $C_3=C_2\cup l$ (spanning $H$) for a smooth conic $C_2$ meeting a line $l$ at a (smooth) point $x_1$. Take now a general point $x_2\in C_2\cap S_{sm}$ (so that a fortiori $x_2\in X_{sm}$). There exists 6 lines on $X$ passing through $x_2$; as on a normal cubic surface with at most RDPs we can not have 6 lines all passing through a smooth point, there always exists a line $l'\subset X$ trough $x_2$ that is not contained in $H$. 
The curve $C_4= C_2\cup l \cup l'$ is a non-degenerate rational quartic curve on $X$ that moreover satisfies $\mult_p(C_4)\le 1, \ \forall p\in \Sing(X)$.

\subsubsection*{The singlar case: construction of the scroll $\Sigma$}\label{construction-scroll}
Let $C_4=C_2\cup l\cup l'$ be a non-degenerate rational quartic on $X$ as above. $l$ and $l'$ are two skew lines meeting the conic $C_2$ in two smooth points $x_1$ and $x_2$, respectively.
We \DM{explain} here how to obtain the scroll $\Sigma$ from $C_4$. 
We fix two points $y_1\in l, \ y_2 \in l'$,\mbox{$ y_i \ne x_i, \ i=1,2$}; we call $D$ the line $D:=\overline{y_1y_2}$.
We take then two points $x_3, \ y_3$ with $x_3 \in C_2$ and $y_3\in D$.
We join with a line all the points $z \in D, \ z'\in C_2$ satisfying:
\begin{equation}\label{cross-ratio}
[x_1,x_2,x_3,z']=[y_1,y_2,y_3,z]. 
\end{equation}
The lines obtained in this way form the ruling of a cubic scroll $\Sigma$. In order to see that we are effectively drawing a cubic scroll we observe that having fixed the 3-uples of points $x_1,x_2,x_3,$ and $y_1,y_2,y_3$, we define a isomorphism $\phi:D\to C_2$ such that $\forall z\in D, \ \phi(z)$ is the only point $z'$ on $C_2$ satisfying the cross ratio equation \ref{cross-ratio}. $\Sigma$ is the join of this isomorphism.
The scroll $\Sigma$ has the line $D$ as directrix; $l$ and $l'$ are lines in the ruling of $\Sigma$ so that $l\sim l'\sim F$ in $\Pic(\Sigma)$.
Computing the intersections, we see that $C_2\sim D+F$ ($C_2$ lies in a plane orthogonal to $D$) so that $C_4 \sim D+3F$ (we have indeed $C_4 \cdot D=2$, given by the two points $y_1, y_2$).

\vspace{2mm}
We can now prove theorem \ref{normal_quintic}.
\begin{proof}[Proof of theorem \ref{normal_quintic}]
	Let $X$ be a normal cubic threefold without triple points. Since we have already proved theorem \ref{smooth_quintic}, we 
	still need to treat the cases where $X$ is singular along a line or a conic pf second type or along a rational quartic. 
	We consider then $C_4$ a rational quartic on $X$ and $\Sigma$ the cubic scroll built from $C_4$ as illustrated above. By construction, the curve $\C$, residual to $C_4$ in $X\cap \Sigma$, belongs to the linear equivalence class $2D+3F$ in $\Pic(\Sigma)$.
	Applying theorem \ref{AG_anticanonical}, we deduce that $\C$ is an AG curve with $\omega_{\C}\simeq \mathcal{O}_{\C}$ hence an AG quintic curve of arithmetic genus 1. Moreover, from its construction, we know that in $\DM{\Pic\,}(\Sigma)$,  $\C \sim -K_{\Sigma}\sim 2D + 3F$. Consider now the short exact sequence of $\mathcal{O}_{\Sigma}$-modules:
	$$0\longrightarrow \mathcal{O}_{\Sigma}(K_{\Sigma})\longrightarrow \mathcal{O}_{\Sigma}\longrightarrow \mathcal{O}_{\C}\longrightarrow 0.$$
	Taking the long exact sequence in cohomology we get:
	
	$$ 0\longrightarrow H^0(\mathcal{O}_{\Sigma}(K_{\Sigma}))\longrightarrow H^0(\mathcal{O}_{\Sigma})\longrightarrow H^0(\mathcal{O}_{\C})\longrightarrow H^1(\mathcal{O}_{\Sigma}(K_{\Sigma}))\longrightarrow\cdots .$$
	
	As $K_{\Sigma}=-2D-3F$ is not effective, $h^0(\mathcal{O}_{\Sigma}(K_{\Sigma}))=h^0(\mathcal{O}_{\Sigma}(-2D-3F))=0$ and by 
	Serre's duality  $h^1(\mathcal{O}_{\Sigma}(K_{\Sigma}))= h^1(\mathcal{O}_{\Sigma})=0$.
	This implies that the restriction map \mbox{$H^0(\mathcal{O}_{\Sigma})\to H^0(\mathcal{O}_{\C})$} is an isomorhism, hence $H^0(\mathcal{O}_{\C})=\D{C}$ and in particular $1\in H^0(\mathcal{O}_{\C})$ is nowhere vanishing. 
\end{proof}

We can now prove theorem \ref{Ulrich-normal}
\begin{proof}[Proof of Thm. \ref{Ulrich-normal}]

	Again, we can restrict to the case where $X$ is either the secant threefold or singular along a line or a conic of second type.
	By theorem \ref{normal_quintic}, we get the existence of a normal AG elliptic quintic $\C\subset X$. $\C$ satisfies $H^0(\mathcal{O}_{\C})=\D{C}$   
	and the class $1 \in H^0(\mathcal{O}_{\C})$ does not vanish on any point of $\C$ so that $1_x$ generates \mbox{$\mathcal{O}_{\C,x}\simeq \DM{\Ext}^1(\mathcal{I}_{\mathcal{C},x}, \mathcal{O}_{X,x})$} 
	forall $x \in \C$. 
	Arguing exactly as in the proof of proposition \ref{Ulrich_smooth}, we can assert that, by \ref{pd_1} and \ref{Serre}, the ACM sheaf $\mathcal{N}$ obtained from $\C$ by Serre's correspondence is locally free.
	By the non-degeneracy of $\C$ and by remark \ref{Ulrich_lf} we can conclude that $\E:=\mathcal{N}(2)$ is a skew-symmetric Ulrich bundle.
\end{proof}

As a direct consequence of theorem \ref{Ulrich-normal} we get:

\begin{thm}\label{normal_Pfaff} Let $X$ be a normal cubic threefold that is not a cone. Then $X$ is Pfaffian.
\end{thm}
\

The last sections of the chapter are devoted to the study of Pfaffian representations of the cubic hypersurfaces we still need to consider, namely:
\begin{itemize}
	\item Cubic threefolds that are cones over cubic hypersurfaces in $\D{P}^r, r\le 3$.
	\item \DM{Non-normal} cubic threefolds.
\end{itemize}
\section{Non normal cubic threefolds and cones}
\subsubsection*{Cubic threefolds presenting triple points} 

\begin{prop}Suppose that $X$ is a cone over a cubic hypersurface in $\D{P}^r, \  r\le 3$.
	Then $X$ is Pfaffian.
\end{prop}
\begin{proof}
	If a cubic threefold $X\subset \D{P}(V)\simeq \D{P}^4$ is a cone, there exists then a vector subspace $U < V^*$ of dimension $r+1 \le 4$ such that $F \in S^3(U)$. $X$ has then multiplicity 3 along the linear space $\D{P}(U^{\perp})\subset X, \ \D{P}(U^{\perp})\simeq \D{P}^{3-r}$. 
	The polynomial $F$ defines an \DM{$(r-1)$-dimensional} cubic hypersurface $Y$ and it is clear that the study of Pfaffian representations of $X$ reduces to the study of Pfaffian representations of $Y$. 
	Since it is well known that a cubic hypersurface of dimension \DM{less than} or equal to \DM{two} is always Pfaffian (see for example \cite{Beauville},  \cite{FM}) $X$ is Pfaffian too. 
\end{proof}

\DM{\subsubsection*{Non-normal cubic threefolds}}

\

We still have to analyse the case where $X$ is not normal, namely when $\DM{\Sing}(X)$ has codimension \DM{less than} 2.
We prove the following:
\begin{prop}\label{non-normal} Let $X$ be a non-normal cubic threefold. Then $X$ is Pfaffian.
\end{prop}

Whenever $\DM{\codim}(\DM{\Sing}(X))\le 1$, one of the following occurs:
\begin{itemize}
	\item $\DM{\codim}(\DM{\Sing}(X))=1$, hence $\DM{\Sing}(X)$ contains a surface $S$. From the fact that $X$ contains all the lines generated by points in $S$ we deduce that $S$ must be a plane.
	\item $\DM{\codim}(\DM{\Sing}(X))=0$, hence $X$ is a non-integral cubic threefold, given by the union of a 3-plane $H$ and of a quadric hypersurface $Q, \ 1\le \DM{\rk}(Q)\le 5$.  
\end{itemize}
For each case we will write explicitely a skew-symmetric matrix of linear forms whose Pfaffian defines the threefold we are considering. 
\paragraph{Cubic threefolds singular along a plane}

\

Suppose that $X$ contains a double plane $\Delta$ supported on $\{X_3=0, X_4=0\}$. 
Up to an appropriate change of coordinates on $\D{P}^4$, the polynomial $F$ has the form:

$$F(X_0,\dots X_4)= X_0 X_3^2 + X_1 X_4^2 + X_2X_3X_4.$$

The skew-symmetric matrix $M$ defined as:
\begin{equation*}M=
\begin{pmatrix} 0 & X_3 & X_4& 0&0& X_2\\
-X_3& 0&0&0&X_4&0\\
-X_4&0&0&X_3&0&0\\
0&0&-X_3&0&0& X_1\\
0&-X_4&0&0&0&X_0\\
-X_2&0&0&-X_1&-X_0&0
\end{pmatrix} 
\end{equation*}

satisfies $\DM{\Pf}(M)=F$ providing then a Pfaffian representation of $X$. 
\paragraph{Non-integral cubic threefolds}

\

Whenever $X:=\{F=0\}$ is not integral, the polynomial $F$ factors as $$F(X_0,\dots ,X_4)=l(X_0, \dots X_4)Q(X_0,\dots X_4),$$ \DM{where $l$ is} a linear form and $Q$ \DM{is} a quadratic form of rank less than or equal to five.
It is well know that a quadric $Q$ of rank $\DM{\rk}(Q)\le 5$ is Pfaffian, so that we can find a $4\times 4$ skew-symmetric matrix of linear forms $M_Q$ such that $\DM{\Pf}(M_Q)=Q(X_0,\dots X_4)$ \DM{(such} a matrix can be easily written down \DM{explicitly}, otherwise we refer to \cite{Beauville}\DM{).}
We construct then $M$, a $6\times 6$ \DM{skew-symmetric} matrix of linear forms such that $\DM{\Pf}(M)=F$ as follows:

\begin{equation*}
M=\left(
\begin{array}{cr}
M_Q & \begin{matrix} 0 & 0 \\ 0&0\\0&0\\ 0&0 \end{matrix} \\
\begin{matrix} 0&0&0&0\\ 0&0&0&0 \end{matrix} & \begin{matrix} \hspace{1mm} 0 & l \\ -l& 0 \end{matrix}
\end{array}
\right)
\end{equation*}
This complete the proof of proposition \ref{non-normal} and leads us, finally, to our main result:
\begin{thm}\label{teoremone_Pf} A cubic threefold $X\subset \D{P}^4$ always admits a Pfaffian representation.
\end{thm}

\end{document}